\documentclass{siamltex}%
\usepackage{amsfonts}
\usepackage{amsmath}
\usepackage{amssymb}
\usepackage{graphicx}
\usepackage{enumerate}
\usepackage{color}
\usepackage{hyperref}%
\newtheorem{example}[theorem]{Example}

\newtheorem{remark}[theorem]{Remark}

\begin{document}

\title{Chain controllability of linear control systems}
\author{Fritz Colonius\\Institut f\"{u}r Mathematik, Universit\"{a}t Augsburg, Augsburg, Germany
\and Alexandre J. Santana and Eduardo C. Viscovini\\Departamento de Matem\'{a}tica, Universidade Estadual de Maring\'{a}\\Maring\'{a}, Brazil}
\maketitle

\begin{center}

\end{center}

\textbf{Abstract: }For linear control systems with bounded control range,
chain controllability properties are analyzed. It is shown that there exists a
unique chain control set and that it equals the sum of the control set around
the origin and the center Lyapunov space of the homogeneous part. For the
proof, the linear control system is extended to a bilinear control system on
an augmented state space. This system induces a control system on projective
space. For the associated control flow attractor-repeller decompositions are
used to show that the control system on projective space has a unique chain
control set that is not contained in the equator. It is given by the image of
the chain control set of the original linear control system.



\section{Introduction\label{Section1}}

We study generalized controllability properties of linear control systems on
$\mathbb{R}^{n}$ with control restrictions of the form%
\begin{equation}
\dot{x}(t)=Ax(t)+Bu(t),\,u\in\mathcal{U}, \label{1}%
\end{equation}
where $A\in\mathbb{R}^{n\times n},B\in\mathbb{R}^{n\times m}$, and the set of
control functions is defined by%
\begin{equation}
\mathcal{U}=\{u\in L^{\infty}(\mathbb{R},\mathbb{R}^{m})\left\vert u(t)\in
U\text{ for almost all }t\in\mathbb{R}\right.  \}; \label{U}%
\end{equation}
here the control range $U$ is a compact and convex neighborhood of
$0\in\mathbb{R}^{m}$. We denote the solution for initial condition
$x(0)=x_{0}\in\mathbb{R}^{n}$ and control $u\in\mathcal{U}$ by $\varphi
(t,x_{0},u),t\in\mathbb{R}$.

The main topic of this paper are chain controllability properties (cf.
Definition \ref{Definition_chain_control}). They constitute a weaker version
of approximate controllability (in infinite time) and may be difficult to
distinguish from it in numerical computations. Here small jumps in the
trajectories are allowed, and hence chain controllability is not a physical
notion. In the theory of dynamical systems analogous constructions (going back
to Rufus Bowen and Charles Conley) have been quite successful in order to
describe the limit behavior as time tends to infinity for complicated flows.

Due to the control restriction, the familiar linear algebra tools of linear
systems theory have to be complemented by methods from dynamical systems
theory. An important tool for us is the extension of the linear control system
on $\mathbb{R}^{n}$ to a bilinear control system on $\mathbb{R}^{n}%
\times\mathbb{R}$ and an ensuing projection to projective space $\mathbb{P}%
^{n}$. This provides us with a compactification of the state space and at the
same time introduces convenient linear structures. In fact, the affine control
flow $\Phi:\mathbb{R}\times\mathcal{U}\times\mathbb{R}^{n}\rightarrow
\mathcal{U}\times\mathbb{R}^{n},\Phi_{t}(u,x_{0})=(u(t+\cdot),\varphi
(t,x_{0},u))$ on the state space $\mathcal{U}\times\mathbb{R}^{n}$ associated
with linear control system (\ref{1}) is extended to a linear control flow
$\Phi^{1}$ on the vector bundle $\mathcal{U}\times\mathbb{R}^{n}%
\times\mathbb{R}$ over the base space $\mathcal{U}$. The projection to
$\mathbb{P}^{n}$ allows us to use attractor-repeller decompositions which
clarify the relation between the chain control set (the maximal chain
controllable subset) and the control set (the maximal approximately
controllable subset) in $\mathbb{R}^{n}$. Furthermore, the powerful Selgrade
decomposition can be applied to the linear control flow $\Phi^{1}$ revealing
further relations between the central Selgrade bundle $\mathcal{V}_{c}^{1}$
(i.e., the unique Selgrade bundle where the last component in $\mathcal{U}%
\times\mathbb{R}^{n}\times\mathbb{R}$ is nontrivial) and the chain control set
on $\mathbb{P}^{n}$. In the theory of nonlinear differential equations
analogous extensions of the state space and ensuing projections to the sphere
bear the name of Poincar\'{e} sphere; cf. Perko \cite{Perko}. The Selgrade
decomposition of linear flows on vector bundles (cf., e.g., Salamon and
Zehnder \cite{SalZ88}, Colonius and Kliemann \cite[Theorem 5.2.5]{ColK00})
provides the finest decomposition into exponentially separated subbundles. It
is constructed by the finest Morse decomposition of the induced flow on the
projective bundle $\mathcal{U}\times\mathbb{P}^{n}$, where the Morse sets are
the maximal chain transitive subsets.

Many techniques and results for chain controllability rely on compactness
properties. Hence a major source of difficulties in our case is the fact that
the state space $\mathbb{R}^{n}$ is not compact. The construction using the
Poincar\'{e} sphere admits the description of the system behavior near infinity.

For chain transitivity and chain recurrence in the topological theory of flows
on metric spaces cf. Alongi and Nelson \cite{AlonN07}. Background on control
sets, chain control sets, and control flows is given in Colonius and Kliemann
\cite{ColK00, ColK14}, Kawan \cite{Kawan13}, and also in Desheng Li
\cite{Li07}, who proposes an approach to chain controllability via
differential inclusions. A recent paper \cite{DaSilva23} by da Silva analyzes
compact chain control sets of linear control systems on connected Lie groups.
Ayala, da Silva, and Mamani \cite{AydSM23} explicitly compute control sets for
linear control systems in two dimensional cases. In Colonius and Santana
\cite{ColS23} the approach of the present paper is used for affine flows on
vector bundles and some applications to affine control systems are given.

The main results of this paper are Theorem \ref{Theorem_ccs1} and Theorem
\ref{Theorem_central}. The former theorem characterizes the chain control set
$E$ as the sum of the center Lyapunov space for $A$ (i.e., the sum of the
generalized eigenspaces for eigenvalues with vanishing real part) and the
closure of the control set around the origin. On the way, we need that the
chain recurrent set of the flow for a linear autonomous differential equation
coincides with the center Lyapunov space; cf. Theorem \ref{Theorem_1_Eduardo}.
This is presumably known, but we could not find a reference in the literature,
hence we include a proof. Theorem \ref{Theorem_central} characterizes the
central Selgrade bundle $\mathcal{V}_{c}^{1}$. A consequence presented in
Corollary \ref{Corollary_ccs} is that the chain control set $E_{c}^{1}$ on the
projective Poincar\'{e} sphere $\mathbb{P}^{n}$ coincides with the closure of
the image of the chain control set $E$.

Section \ref{Section2} contains preliminary results on control systems and
recalls notions and results from the topological theory of flows on metric
spaces including Selgrade's theorem for linear flows on vector bundles. It is
noteworthy that we can simplify the notion of chain control sets: It is not
necessary to require that for every $x\in E$ there exists $u\in\mathcal{U}$
with $\varphi(t,x,u)\in E$ for all $t\in\mathbb{R}$, since this holds for
every maximal chain controllable set; cf. Remark \ref{Remark_invariance}.
Section \ref{Section3} proves that the center Lyapunov space of an autonomous
linear differential equation is chain transitive yielding first results on
chain controllability. Section \ref{Section4} provides the characterization of
the chain control set $E$ in $\mathbb{R}^{n}$. Section \ref{Section5} analyzes
the induced system on the Poincar\'{e} sphere, and Section \ref{Section6}
presents two examples.

\textbf{Notation:} The control $u(t)\equiv0$ in $\mathcal{U}$ is denoted by
$0_{\mathcal{U}}$. The trivial subspace $\{0\}\subset\mathbb{R}^{n}$ is
denoted by $0_{n}$ and $0_{1}$ is abbreviated by $0$. The letter $\mathbb{K}$
denotes $\mathbb{R}$ or $\mathbb{C}$.

\section{Preliminaries\label{Section2}}

The first subsection presents elements of the topological theory of flows on
metric spaces. Subsection \ref{Subsection2.2} describes control sets and chain
control sets for control-affine systems and recalls the notion of control
flows. Subsection \ref{Subsection2.3} derives some properties of control sets
for linear control systems.

\subsection{Flows on metric spaces\label{Subsection2.1}}

For the following concepts for flows on metric spaces cf. Alongi and Nelson
\cite{AlonN07} and Colonius and Kliemann \cite[Appendix B]{ColK00},
\cite{ColK14}.

A conjugacy of flows $\Psi$ and $\Psi^{\prime}$ on complete metric spaces $X$
and $X^{\prime}$, respectively, is a homeomorphism $h:X\rightarrow X^{\prime}$
with $h(\Psi(t,x))=\Psi^{\prime}(t,h(x))$ for all $x\in X,t\in\mathbb{R}$.
Where convenient, we also write $\Psi_{t}(x)=\Psi(t,x)$.

For $\varepsilon,\,T>0$ an $(\varepsilon,T)$-chain $\zeta$ for $\Psi$ from $x$
to $y$ is given by $k\in\mathbb{N}\mathbf{,}$ $T_{0},\ldots,T_{k-1}\geq T$,
and $x_{0}=x,\ldots,x_{k-1},x_{k}=y\in X$ with $d(\Psi(T_{i},x_{i}%
),x_{i+1})<\varepsilon$ for $i=0,\ldots,k-1$. For $x\in X$ the $\omega$-limit
set is $\omega(x)=\{y\in X\left\vert \exists t_{k}\rightarrow\infty:\Psi
(t_{k},x)\rightarrow y\right.  \}$ and the $\alpha$-limit set is
$\alpha(x)=\{y\in X\left\vert \exists t_{k}\rightarrow-\infty:\Psi
(t_{k},x)\rightarrow y\right.  \}$. The (forward) chain limit set for $\Psi$
is $\Omega(x)=\{y\in X\left\vert \forall\varepsilon,T>0\,\exists
(\varepsilon,T)\text{-chain from }x\text{ to }y\right.  \}$. The backward
chain limit set $\Omega^{\ast}(x)$ of $x$ is the chain limit set of $x$ for
the time reversed flow $\Psi^{\ast}(t,x):=\Psi(-t,x),t\in\mathbb{R},x\in X$. A
point $x\in X$ is chain recurrent if $x\in\Omega(x)$, and a nonvoid set $Y$ is
called chain transitive if $y\in\Omega(x)$ for all $x,y\in Y$. Observe that
any nonvoid subset of a chain transitive set is chain transitive, and (cf.
\cite[Proposition 2.7.10]{AlonN07}) a set is chain transitive if and only if
its closure is chain transitive. If $X$ is chain transitive for a flow on $X$,
then also the flow is called chain transitive. On a compact metric space the
maximal chain transitive sets, called the chain recurrent components, are the
connected components of the chain recurrent set (i.e., the set of chain
recurrent points) and the flow restricted to a chain recurrent component is
chain transitive.

\begin{proposition}
\label{LemCLSAtReverseTime}For any flow $\Psi$ on $X$ the forward and backward
chain limit sets are related by $y\in\Omega(x)$ if and only if $x\in
\Omega^{\ast}(y)$.
\end{proposition}

\begin{proof}
Let $y\in\Omega(x)$ and $\varepsilon,T>0$. The difficulty in the proof lies in
the fact, that in the last step of an $(\varepsilon,T)$-chain from $x$ to $y$
there is a jump to $y$. By continuity of the flow there is $\delta
\in(0,\varepsilon)$ such that $d(x,z)<\delta$ implies $d(\Psi(-T,x),\Psi
(-T,z))<\varepsilon$. Consider a $(\delta,2T)$ chain from $x$ to $y$ given by
$T_{0},\ldots,T_{k-1}\geq2T,x=x_{0},x_{1},\ldots,x_{k}=y$. We construct an
$(\varepsilon,T)$-chain from $y$ to $x$ for the time reversed flow $\Psi
^{\ast}$ in the following way:

Let $T_{0}^{\ast}=T,T_{1}^{\ast}=T_{k-1}-T\geq T,T_{2}^{\ast}=T_{k-2}%
,\ldots,T_{k}^{\ast}=T_{0}$ and%
\[
x_{0}^{\ast}=y,\,x_{1}^{\ast}=\Psi_{-T}\left(  \Psi_{T_{k-1}}(x_{k-1})\right)
,\,x_{2}^{\ast}=\Psi_{T_{k-2}}(x_{k-2}),\ldots,\,x_{k}^{\ast}=\Psi_{T_{0}%
}(x_{0}),\,x_{k+1}^{\ast}=x.
\]
Since $d(y,\Psi_{T_{k-1}}(x_{k-1}))<\delta$ it follows that
\begin{align*}
d\left(  \Psi_{T_{0}^{\ast}}^{\ast}(x_{0}^{\ast}),x_{1}^{\ast}\right)   &
=d\left(  \Psi_{-T}(y),\Psi_{-T}\left(  \Psi_{T_{k-1}}(x_{k-1})\right)
\right)  <\varepsilon,\\
d\left(  \Psi_{T_{i}^{\ast}}^{\ast}(x_{i}^{\ast}),x_{i+1}^{\ast}\right)   &
=d\left(  x_{k-i},\Psi_{T_{k-i-1}}(x_{k-i-1})\right)  <\varepsilon\text{ for
}i\in\{1,\ldots,k-1\},\\
d\left(  \Psi_{T_{k}^{\ast}}^{\ast}(x_{k}^{\ast}),x_{k+1}^{\ast}\right)   &
=d\left(  \Psi_{-T_{0}}(\Psi_{T_{0}}(x_{0})),x_{0}\right)  =0.
\end{align*}
Thus $x\in\Omega^{\ast}(y)$. The other implication follows analogously since
$\Psi=(\Psi^{\ast})^{\ast}$.
\end{proof}

\begin{proposition}
\label{Proposition_flow_inv}(i) For a flow $\Psi$ on $X$ the chain limit sets
$\Omega(x)$ and $\Omega^{\ast}(x),x\in X$, are invariant.

(ii) Any maximal chain transitive set $Y$ of $\Psi$ is invariant and coincides
with $\Omega(x)\cap\Omega^{\ast}(x)$ for all $x\in Y$.
\end{proposition}

\begin{proof}
(i) By Proposition \ref{LemCLSAtReverseTime} it suffices to prove that
$\Psi_{t}(y)\in\Omega(x)$ for all $t\in\mathbb{R}$ and all $y\in\Omega(x)$.
Fix $t\in\mathbb{R}$ and let $\varepsilon,T>0$. By continuity there is
$\delta>0$ such that
\[
d(z,y)<\delta\text{ implies }d\left(  \Psi_{t}(z),\Psi_{t}(y)\right)
<\varepsilon.
\]
Pick a $(\delta,S)$ chain from $x$ to $y$ with $S+t\geq T$. Then the final
piece of this chain satisfies $d\left(  \Psi_{T_{k-1}}(x_{k-1}),y\right)
<\delta$ implying
\[
d\left(  \Psi_{t+T_{k-1}}(x_{k-1}),\Psi_{t}(y)\right)  =d\left(  \Psi_{t}%
(\Psi_{T_{k-1}}(x_{k-1})),\Psi_{t}(y)\right)  <\varepsilon.
\]
Hence we obtain an $(\varepsilon,T)$-chain from $x$ to $\Psi_{t}(y)$.

(ii) By (i) it suffices to prove that $Y=\Omega(x)\cap\Omega^{\ast}(x)$ for
all $x\in Y$. Let $y\in Y$. Then $y\in\Omega(x)$ and $x\in\Omega(y)$, hence by
Proposition \ref{LemCLSAtReverseTime} it follows that $y\in\Omega^{\ast}(x)$.
For the converse inclusion note that any point $y\in\Omega(x)\cap\Omega^{\ast
}(x)$ satisfies $y\in\Omega(x)$ and $x\in\Omega(y)$, hence the chain
transitive set $\Omega(x)\cap\Omega^{\ast}(x)$ is contained in the maximal
chain transitive set containing $x$.
\end{proof}

An attractor is a compact invariant set $\mathcal{A}$ such that $\mathcal{A}%
\subset\mathrm{int}N$ for a set $N$ with%
\begin{equation}
\mathcal{A}=\omega(N):=\left\{  y\in X\left\vert \exists t_{k}\rightarrow
\infty~\exists x_{k}\in N:\Psi_{t_{k}}(x_{k})\rightarrow y\right.  \right\}  .
\label{A}%
\end{equation}
The complementary repeller is $\mathcal{A}^{\ast}:=\left\{  y\in X\left\vert
\omega(x)\cap\mathcal{A}=\varnothing\right.  \right\}  $. It is also compact
invariant and has the property that $\mathcal{A}^{\ast}\subset\mathrm{int}%
N^{\ast}$ for a set $N^{\ast}$ with%
\begin{equation}
\mathcal{A}^{\ast}=\omega^{\ast}(N^{\ast}):=\left\{  y\in X\left\vert \exists
t_{k}\rightarrow-\infty~\exists x_{k}\in N:\Psi_{t_{k}}(x_{k})\rightarrow
y\right.  \right\}  . \label{repeller}%
\end{equation}
The relation to the chain recurrent set is given by the following theorem (cf.
\cite[Theorem B.2.26]{ColK00}).

\begin{theorem}
\label{Theorem_attractor}For a flow on a compact metric space $X$ the chain
recurrent set coincides with $\bigcap\left(  \mathcal{A}\cup\mathcal{A}^{\ast
}\right)  $, where the intersection is taken over all attractors $\mathcal{A}$.
\end{theorem}

A related concept are Morse decompositions introduced next. Note first that a
compact subset $K\subset X$ is called isolated invariant for $\Psi$, if
$\Psi_{t}(x)\in K$ for all $x\in K$ and all $t\in\mathbb{R}$ and there exists
a set $N$ with $K\subset\mathrm{int}\,N$ such that $\Psi_{t}(x)\in N$ for all
$t\in\mathbb{R}$ implies $x\in K$.

\begin{definition}
\label{defMorsedecomp}A Morse decomposition of a flow $\Psi$ on a compact
metric space $X$ is a finite collection $\left\{  \mathcal{M}_{i}\left\vert
i=1,\ldots,\ell\right.  \right\}  $ of nonvoid, pairwise disjoint, and compact
isolated invariant sets such that

(i) for all $x\in X$ the limit sets satisfy $\omega(x),\,\alpha(x)\subset
\bigcup_{i=1}^{\ell}\mathcal{M}_{i}$, and

(ii) suppose there are $\mathcal{M}_{j_{0}},\mathcal{M}_{j_{1}},\ldots
,\mathcal{M}_{j_{k}\text{ }}$ and $x_{1},\ldots,x_{k}\in X\setminus
\bigcup_{i=1}^{\ell}\mathcal{M}_{i}$ with $\alpha(x_{i})\subset\mathcal{M}%
_{j_{i-1}}$ and $\omega(x_{i})\subset\mathcal{M}_{j_{i}}$ for $i=1,\ldots,k$;
then $\mathcal{M}_{j_{0}}\neq\mathcal{M}_{j_{k}}$.

The elements of a Morse decomposition are called Morse sets. A Morse
decomposition $\left\{  \mathcal{M}_{1},\ldots,\mathcal{M}_{k}\right\}  $ is
called \textit{finer} than a Morse decomposition $\left\{  \mathcal{M}%
_{1}^{\prime},\ldots,\mathcal{M}_{k^{\prime}}^{\prime}\right\}  $, if for all
$j\in\left\{  1,\ldots,k^{\prime}\right\}  $ there is $i\in\left\{
1,\ldots,k\right\}  $ with $\mathcal{M}_{i}\subset\mathcal{M}_{j}^{\prime}$.
\end{definition}

An order is defined by the relation $\mathcal{M}_{i}\preceq\mathcal{M}_{j}$ if
there are indices $j_{0},\ldots,j_{k}$ with $\mathcal{M}_{i}=\mathcal{M}%
_{j_{0}},\mathcal{M}_{j}=\mathcal{M}_{j_{k}}$ and points $x_{j_{i}}\in X$
with
\[
\alpha(x_{j_{i}})\subset\mathcal{M}_{j_{i}-1}\mbox{ and }\omega(x_{j_{i}%
})\subset\mathcal{M}_{j_{i}}\mbox{ for }i=1,\ldots,k.
\]

The following theorem relates chain recurrent components and Morse
decompositions; cf. \cite[Theorem 8.3.3]{ColK14}.

\begin{theorem}
For a flow on a compact metric space there exists a finest Morse decomposition
if and only if the chain recurrent set has only finitely many connected
components. Then the Morse sets coincide with the chain recurrent components.
\end{theorem}

\begin{remark}
\label{Remark_Hurley}For continuous maps $f$ on locally compact metric spaces
$X$, Hurley \cite{Hurley92} proposes to modify the definition of chain
recurrence by considering instead of a constant $\varepsilon>0$ continuous
functions $\varepsilon:X\rightarrow(0,\infty)$. A strong $\varepsilon(\cdot
)$-chain consists of $x_{0},x_{1},\ldots,x_{k}$ with $d(f(x_{j}),x_{j+1}%
)<\varepsilon(f(x_{j}))$ for all $j$ and a point $x\in X$ is strongly chain
recurrent if for each such function $\varepsilon(\cdot)$ there is a strong
$\varepsilon(\cdot)$-chain from $x$ to $x$. This approach does not appear
appropriate in our context, mainly due to the fact that we will use the
compactification provided by the Poincar\'{e} sphere $\mathbb{P}^{n}$.
\end{remark}

We will consider vector bundles $\mathcal{V}=B\times\mathbb{R}^{n}$, where $B$
is a compact metric base space. A linear flow $\Psi=(\theta,\psi)$ on
$B\times\mathbb{R}^{n}$ is a flow of the form%
\[
\Psi:\mathbb{R}\times B\times\mathbb{R}^{n}\rightarrow B\times\mathbb{R}%
^{n},\,\Psi_{t}(b,x)=(\theta_{t}b,\psi(t,b,x))\text{ for }(t,b,x)\in
\mathbb{R}\times B\times\mathbb{R}^{n},
\]
where $\theta$ is a flow on the base space $B$ and $\psi(t,b,x)$ is linear in
$x$, i.e. $\psi(t,b,\alpha_{1}x_{1}+\alpha_{2}x_{2})=\alpha_{1}\psi
(t,b,x_{1})+\alpha_{2}\psi(t,b,x_{2})$ for $\alpha_{1},\alpha_{2}\in
\mathbb{R}$ and $x_{1},x_{2}\in\mathbb{R}^{n}$. A closed subset $\mathcal{V}$
of $B\times\mathbb{R}^{n}$ that intersects each fiber $\{b\}\times
\mathbb{R}^{n},b\in B$, in a linear subspace of constant dimension is a
subbundle. Denote the projection $\mathbb{R}^{n}\setminus\{0\}\rightarrow
\mathbb{P}^{n-1}$ as well as the corresponding map $B\times(\mathbb{R}%
^{n}\setminus\{0\})\rightarrow B\times\mathbb{P}^{n-1}$ by the letter
$\mathbb{P}$. A linear flow $\Psi$ induces a flow $\mathbb{P}\Psi$ on the
projective bundle $B\times\mathbb{P}^{n-1}$, which is a compact metric space;
a metric on $\mathbb{P}^{n-1}$ can be obtained by defining, for elements
$p_{1}=\mathbb{P}x,p_{2}=\mathbb{P}y$,%
\begin{equation}
d(p_{1},p_{2})=\min\left\{  \left\Vert \frac{x}{\left\Vert x\right\Vert
}-\frac{y}{\left\Vert y\right\Vert }\right\Vert ,\left\Vert \frac
{x}{\left\Vert x\right\Vert }+\frac{y}{\left\Vert y\right\Vert }\right\Vert
\right\}  , \label{metric_P}%
\end{equation}
and a metric on $B\times\mathbb{P}^{n-1}$ is defined by taking the maximum of
the distances in $B$ and $\mathbb{P}^{n-1}$.

For a linear flow $\Psi$ two nontrivial invariant subbundles $(\mathcal{V}%
^{+},\mathcal{V}^{-})$ with $B\times\mathbb{R}^{n}=\mathcal{V}^{+}%
\oplus\mathcal{V}^{-}$ (a Whitney sum) are exponentially separated if there
are $c,\mu>0$ with%
\begin{equation}
\left\Vert \Psi(t,b,x^{+})\right\Vert \leq ce^{-\mu t}\left\Vert
\Psi(t,b,x^{-})\right\Vert ,t\geq0,\text{ for }(b,x^{\pm})\in\mathcal{V}^{\pm
},\,\left\Vert x^{+}\right\Vert =\left\Vert x^{-}\right\Vert . \label{exp_sep}%
\end{equation}
The following is Selgrade's theorem for linear flows; cf. \cite[Theorem
9.2.5]{ColK14}, and \cite[Theorem 5.1.4]{ColK00} for the result on exponential separation.

\begin{theorem}
\label{Theorem_Selgrade}Let $\Psi=(\theta,\psi):\mathbb{R}\times
B\times\mathbb{R}^{n}\rightarrow B\times\mathbb{R}^{n}$ be a linear flow on
the vector bundle $B\times\mathbb{R}^{n}$ with chain transitive flow $\theta$
on the base space $B$.\ Then the projected flow $\mathbb{P}\Psi$ on
$B\times\mathbb{P}^{n-1}$ has the linearly ordered chain recurrent components
$\mathcal{M}_{1}\preceq\cdots\preceq\mathcal{M}_{\ell},1\leq\ell\leq n$. These
components form the finest Morse decomposition for $\mathbb{P}\Psi$. The lifts
of the Morse sets $\mathcal{M}_{i}$
\[
\mathcal{V}_{i}=\mathbb{P}^{-1}\mathcal{M}_{i}:=\left\{  (b,x)\in
B\times\mathbb{R}^{n}\left\vert x\not =0\Rightarrow(b,\mathbb{P}%
x)\in\mathcal{M}_{i}\right.  \right\}  ,
\]
are subbundles, called the Selgrade bundles. They form the decomposition
$B\times\mathbb{R}^{n}=\mathcal{V}_{1}\oplus\cdots\oplus\mathcal{V}_{\ell}$.
This Selgrade decomposition is the finest decomposition into exponentially
separated subbundles: For any pair of exponentially separated subbundles
$(\mathcal{V}^{+},\mathcal{V}^{-})$ there is $1\leq j<\ell$ with
\[
\mathcal{V}^{+}=\mathcal{V}_{1}\oplus\cdots\oplus\mathcal{V}_{j}%
\;\text{and\ }\mathcal{V}^{-}=\mathcal{V}_{j+1}\oplus\cdots\oplus
\mathcal{V}_{\ell}\,.
\]
Conversely, subbundles $\mathcal{V}^{+}$ and $\mathcal{V}^{-}$ defined in this
way are exponentially separated.
\end{theorem}

\subsection{Control sets, chain control sets, and control
flows\label{Subsection2.2}}

Consider control-affine systems of the form
\begin{equation}
\dot{x}(t)=X_{0}(x(t))+\sum_{i=1}^{m}u_{i}(t)X_{i}(x(t)),\,u\in\mathcal{U},
\label{control1}%
\end{equation}
where $X_{0},X_{1},\ldots,X_{m}$ are smooth ($C^{\infty}$-)vector fields on a
smooth manifold $M$ and $\mathcal{U}$ is defined by (\ref{U}). We assume that
for every control $u\in\mathcal{U}$ and every initial state $x(0)=x_{0}\in M$
there exists a unique (Carath\'{e}odory) solution $\varphi(t,x_{0}%
,u),t\in\mathbb{R}$.

For $x\in M$ the controllable set $\mathbf{C}(x)$ and the reachable set
$\mathbf{R}(x)$ are defined as
\begin{align*}
\mathbf{C}(x)  &  =\left\{  y\in M\left\vert \exists u\in\mathcal{U}~\exists
T>0:\varphi(T,y,u)=x\right.  \right\}  ,\\
\mathbf{R}(x)  &  =\left\{  y\in M\left\vert \exists u\in\mathcal{U}~\exists
T>0:y=\varphi(T,x,u)\right.  \right\}  ,
\end{align*}
respectively. The accessibility rank condition is defined as%
\begin{equation}
\dim\mathcal{LA}\left\{  X_{0},X_{1},\ldots,X_{m}\right\}  (x)=\dim M\text{
for all }x\in M; \label{ARC}%
\end{equation}
here $\mathcal{LA}\left\{  X_{0},X_{1},\ldots,X_{m}\right\}  (x)$ is the
subspace of the tangent space $T_{x}M$ corresponding to the vector fields,
evaluated in $x$, in the Lie algebra generated by $X_{0},X_{1},\ldots,X_{m}$.
The controllable sets $\mathbf{C}(x)$ are the reachable sets of the time
reversed system%
\begin{equation}
\dot{x}(t)=-X_{0}(x(t))-\sum_{i=1}^{m}u_{i}(t)X_{i}(x(t)),\,u\in\mathcal{U}.
\label{time-reversed}%
\end{equation}
The following definition introduces sets of complete approximate controllability.

\begin{definition}
\label{Definition_control_sets}A nonvoid set $D\subset M$ is called a control
set of system (\ref{control1}) if it has the following properties: (i) for all
$x\in D$ there is a control $u\in\mathcal{U}$ such that $\varphi(t,x,u)\in D$
for all $t\geq0$, (ii) for all $x\in D$ one has $D\subset\overline
{\mathbf{R}(x)}$, and (iii) $D$ is maximal with these properties, that is, if
$D^{\prime}\supset D$ satisfies conditions (i) and (ii), then $D^{\prime}=D$.
\end{definition}

We recall some properties of control sets; cf. Colonius and Kliemann
\cite[Chp. 3]{ColK00}.

\begin{remark}
\label{Remark2.2}If the intersection of two control sets is nonvoid, the
maximality property (iii) implies that they coincide. If (\ref{ARC}) holds
then by \cite[Lemma 3.2.13(i)]{ColK00} $\overline{D}=\overline{\mathrm{int}%
(D)}$ and $D=\mathbf{C}(x)\cap\overline{\mathbf{R}(x)}$ for all $x\in
\mathrm{int}(D)$ and $\mathrm{int}(D)\subset\mathbf{R}(x)$ for all $x\in D$.
\end{remark}

Next we introduce a notion of controllability allowing for (small) jumps
between pieces of trajectories. Here we fix a metric $d$ on $M$.

\begin{definition}
\label{intro2:defchains}Let $x,y\in M$. For $\varepsilon,T>0$ a controlled
$(\varepsilon,T)$\textit{-chain} $\zeta$ from $x$ to $y$ is given by
$k\in\mathbb{N},\ x_{0}=x,x_{1},\ldots,x_{k}=y\in M,\ u_{0},\ldots,u_{k-1}%
\in\mathcal{U}$, and $t_{0},\ldots,t_{k-1}\geq T$ with
\[
d(\varphi(t_{j},x_{j},u_{j}),x_{j+1})<\varepsilon\text{ }\,\text{for\thinspace
all}\,\,\,j=0,\ldots,k-1.
\]
If for every $\varepsilon,T>0$ there is a controlled $(\varepsilon,T)$-chain
from $x$ to $y$, the point $x$ is chain controllable to $y$. The chain
reachable set from $x\ $and the chain controllable set to $x$ are
\begin{align*}
\mathbf{R}^{c}(x)  &  =\left\{  y\in M\left\vert x\text{ is chain controllable
to }y\right.  \right\}  ,\\
\mathbf{C}^{c}(x)  &  =\left\{  y\in M\left\vert y\text{ is chain controllable
to }x\right.  \right\}  ,
\end{align*}
respectively. If a nonvoid set $F\subset M$ has the property that $x$ is chain
controllable to $y$ for all $x,y\in F$, the set $F$ is said to be chain controllable.
\end{definition}

Note that a set $F$ is chain controllable if and only if $F\subset
\mathbf{R}^{c}(x)$ for all $x\in F$. Observe also that chain controllable sets
are a generalized version of limit sets for time tending to infinity. In
analogy to control sets, we define chain control sets as maximal chain
controllable sets.

\begin{definition}
\label{Definition_chain_control}A nonvoid set $E\subset M$ is called a
\textit{chain control set} of system (\ref{control1}) if for all $x,y\in E$
and $\varepsilon,T>0$ there is a controlled $(\varepsilon,T)$-chain from $x$
to $y$, and $E$ is maximal with this property.
\end{definition}

Since the concatenation of two controlled $(\varepsilon,T)$-chains again
yields a controlled $(\varepsilon,T)$-chain, two chain control sets coincide
if their intersection is nonvoid.

\begin{remark}
\label{Remark_invariance}The definition of chain control sets given e.g. in
Colonius and Kliemann \cite{ColK00} requires in addition to the chain
controllability property that for every $x\in E$ there is a control function
$u\in\mathcal{U}$ such that $\varphi(t,x,u)\in E$ for all $t\in\mathbb{R}$.
Theorem \ref{Theorem_control_inv}\textbf{ }will show that this property holds
for every maximal chain controllable set. Thus it is not necessary to impose
this condition. The paper by Li \cite[Proposition 6.8]{Li07} treats control
systems as particular cases of differential inclusions. Hence the shift \ flow
on the space of control functions plays no role. The definition of chain
control sets (cf. \cite[Definition 6.7]{Li07}) also does not require
invariance. The results on chain transitivity properties and Morse
decompositions apply to compact invariant sets of control systems of the form
$\dot{x}=f(x,u)$, where $f(x,U)$ is convex and the control range $U$ is compact.
\end{remark}

We note the following properties of chain controllability.

\begin{proposition}
\label{Proposition_ccs1}(i) For control system (\ref{control1}) the chain
controllable set $\mathbf{C}^{c}(x)$ to $x$ coincides with the chain reachable
set from $x$ for the time reversed system (\ref{time-reversed}).

(ii) Every chain controllable set $F$ is contained in a chain control set.

(iii) The chain controllable sets $\mathbf{C}^{c}(x)$ and the chain reachable
sets $\mathbf{R}^{c}(x)$ are closed.
\end{proposition}

\begin{proof}
(i)\ This follows using the same arguments as for Proposition
\ref{LemCLSAtReverseTime}.

(ii) Define $E$ as the union of all chain controllable sets $F^{\prime}$
containing $F$. Then $E$ is a maximal chain controllable set, hence a chain
control set containing $F$.

(iii) Let $y\in\overline{\mathbf{R}^{c}(x)}$ and fix $\varepsilon,T>0$. There
is $y_{1}\in\mathbf{R}^{c}(x)$ with $d(y,y_{1})<\varepsilon/2$. The last piece
of a controlled $(\varepsilon/2,T)$-chain from $x$ to $y_{1}$ satisfies%
\[
d(\varphi(T_{k-1},x_{k-1},u_{k-1}),y)\leq d(\varphi(T_{k-1},x_{k-1}%
,u_{k-1}),y_{1})+d(y_{1},y)<\varepsilon/2+\varepsilon/2=\varepsilon.
\]
Thus we obtain a controlled $(\varepsilon,T)$-chain from $x$ to $y$, and hence
$\mathbf{R}^{c}(x)$ is closed. The assertion for $\mathbf{C}^{c}(x)$ follows
using (i).
\end{proof}

The following proposition clarifies the relations between chain control sets
and chain reachable and controllable sets.

\begin{theorem}
\label{Theorem_control_inv}(i) For all $y\in\mathbf{C}^{c}(x)$ and all
$u\in\mathcal{U}$ it follows that $\varphi(t,y,u)\in\mathbf{C}^{c}(x)$ for all
$t<0$. Furthermore, there exists $u\in\mathcal{U}$ such that $\varphi
(t,y,u)\in\mathbf{C}^{c}(x)$ for all $t>0$.

(ii) For all $y\in\mathbf{R}^{c}(x)$ and all $v\in\mathcal{U}$ it follows that
$\varphi(t,y,v)\in\mathbf{R}^{c}(x)$ for all $t>0$. Furthermore, there exists
$v\in\mathcal{U}$ such that $\varphi(t,y,v)\in\mathbf{R}^{c}(x)$ for all $t<0$.

(iii) Let $E$ be a chain control set. Then it follows for all $x\in E$ that
$E=\mathbf{R}^{c}(x)\cap\mathbf{C}^{c}(x)$ and that there exists
$u\in\mathcal{U}$ such that $\varphi(t,x,u)\in E$ for all $t\in\mathbb{R}$.
\end{theorem}

\begin{proof}
(i) Let $y\in\mathbf{C}^{c}(x)$. First note that $\varphi(t,y,u)\in
\mathbf{C}^{c}(x)$ for all $t<0$ and $u\in\mathcal{U}$. In fact, for any
$\varepsilon,T>0$, one gets a controlled $\left(  \varepsilon,T\right)
$-chain from $y$ to $x$ by concatenating the trajectory from $\varphi(t,y,u)$
to $y$ with a controlled $\left(  \varepsilon,T\right)  $-chain from $y$ to
$x$.

In the case of positive time, consider for sequences $\varepsilon
^{i}\rightarrow0$ and $T^{i}\rightarrow\infty$ controlled $\left(
\varepsilon^{i},T^{i}\right)  $-chains $\zeta^{i}$ from $y$ to $x$. Let the
first pieces of the chains $\zeta^{i}$ be given by $\varphi(s,y,u_{1}%
^{i}),s\in\lbrack0,T_{1}^{i}]$, with $T_{1}^{i}\geq T^{i}$. Without loss of
generality, $u_{1}^{i}$ converges to some control $u\in\mathcal{U}$. We claim
that $\varphi(t,y,u)\in\mathbf{C}^{c}(x)$ for all $t>0$.

For the proof fix $t>0$. We construct for all $\varepsilon,T>0$ a controlled
$(\varepsilon,T)$-chain from $\varphi(t,y,u)$ to $x$. By compactness of
$\mathcal{U}$ and continuity there is $\delta>0$ such that%
\[
d(\varphi(t,y,u),z)<\delta\text{ implies }d(\varphi(T,\varphi
(t,y,u),v),\varphi(T,z,v))<\varepsilon\text{ for all }v\in\mathcal{U}.
\]
For $i$ large enough we obtain $\varepsilon_{i}<\varepsilon,T^{i}\geq3T$, and
$d(\varphi(t,y,u_{1}^{i}),\varphi(t,y,u))<\delta$.

Now replace the initial piece of the chain $\zeta^{i}$ by two pieces given by%
\[
x_{0}=\varphi(t,y,u),x_{1}=\varphi(T+t,y,u_{1}^{i})=\varphi(T,\varphi
(t,y,u_{1}^{i}),u_{1}^{i}(t+\cdot))
\]
with times $T_{0}=T,T_{1}=T_{1}^{i}-T-t\geq T$ and controls $u_{1}^{i}%
(t+\cdot)$ and $u_{1}^{i}(T+t+\cdot)$. We find%
\[
d(\varphi(T,x_{0},u_{1}^{i}(t+\cdot)),x_{1})=d(\varphi(T,\varphi
(t,y,u),u_{1}^{i}(t+\cdot)),\varphi(T,\varphi(t,y,u_{1}^{i}),u_{1}^{i}%
(t+\cdot)))<\varepsilon
\]
and
\[
\varphi(T_{1},x_{1},u_{1}^{i}(T+t+\cdot))=\varphi(T_{1}^{i}-T-t,\varphi
(T+t,y,u_{1}^{i}),u_{1}^{i}(T+t+\cdot))=\varphi(T_{1}^{i},y,u_{1}^{i}).
\]
Thus we get a controlled $(\varepsilon,T)$-chain from $\varphi(t,y,u)$ to $x$.

The proof of (ii) follows by time reversal using Proposition
\ref{Proposition_ccs1}(i). It remains to prove (iii). By definition, $E$ is a
maximal chain controllable set. Let $x,y\in E$. Then $y\in\mathbf{R}^{c}(x)$
and $x\in\mathbf{R}^{c}(y)$ or, equivalently, $y\in\mathbf{C}^{c}(x)$ showing
that $E\subset\mathbf{R}^{c}(x)\cap\mathbf{C}^{c}(x)$. For the converse
inclusion note that the same argument shows that $\mathbf{R}^{c}%
(x)\cap\mathbf{C}^{c}(x)$ is a chain controllable set, hence contained in $E$.

Let $y\in E=\mathbf{R}^{c}(x)\cap\mathbf{C}^{c}(x)$. By assertions (i) and
(ii) there are controls $u,v\in\mathcal{U}$ such that $\varphi(t,y,u)\in
\mathbf{C}^{c}(x)$ for all $t<0$ and $\varphi(t,y,v)\in\mathbf{R}^{c}(x)$ for
all $t>0$. Then the control
\[
w(t):=\left\{
\begin{array}
[c]{ccc}%
v(t) & \text{for} & t\leq0\\
u(t) & \text{for} & t>0
\end{array}
\right.
\]
yields $\varphi(t,y,w)\in\mathbf{R}^{c}(x)\cap\mathbf{C}^{c}(x)$ for all
$t\in\mathbb{R}$.
\end{proof}

The control flow associated with the linear control system (\ref{1}) is the
flow on $\mathcal{U}\times\mathbb{R}^{n}$ defined by%
\begin{equation}
\Phi:\mathbb{R}\times\mathcal{U}\times\mathbb{R}^{n}\rightarrow\mathcal{U}%
\times\mathbb{R}^{n},\Phi_{t}(u,x)=(\theta_{t}u,\varphi(t,x,u)),
\label{control_flow}%
\end{equation}
where $\theta_{t}u=u(t+\cdot)$ is the right shift on $\mathcal{U}$. The space
$\mathcal{U}$ is a compact metrizable space with respect to the weak$^{\ast}$
topology of $L^{\infty}(\mathbb{R},\mathbb{R}^{m})$ (we fix such a metric) and
the shift flow $\theta$ is \ continuous; cf. Kawan \cite[Proposition
1.15]{Kawan13}. The flow $\Phi$ is continuous; cf. \cite[Proposition
1.17]{Kawan13}.

The relation between chain control sets and the control flow is explained in
the following theorem.

\begin{theorem}
\label{Theorem_equivalence}Let $\mathcal{E}\subset\mathcal{U}\times M$ be a
maximal chain transitive set for the control flow $\Phi$. Then $\left\{  x\in
M\left\vert \exists u\in\mathcal{U}:(u,x)\in\mathcal{E}\right.  \right\}  $ is
a chain control set. Conversely, if $E\subset M$ is a chain control set, then%
\begin{equation}
\mathcal{E}:=\{(u,x)\in\mathcal{U}\times M\left\vert \varphi(t,x,u)\in E\text{
for all }t\in\mathbb{R}\right.  \} \label{lift_E}%
\end{equation}
is a maximal chain transitive set
\end{theorem}

\begin{proof}
This follows from Colonius and Kliemann \cite[Theorem 4.3.11]{ColK00} or Kawan
\cite[Proposition 1.24(iv)]{Kawan13}. Note that the proofs given there do not
use compactness properties. These results assume that $\mathcal{E}$ is a
maximal invariant chain transitive set and the employed definition of chain
control sets requires that for every $x\in E$ there is a control
$u\in\mathcal{U}$ with $\varphi(t,x,u)\in E$ for all $t\in\mathbb{R}$.
Proposition \ref{Proposition_ccs1} shows that every maximal chain controllable
set has this property and Proposition \ref{Proposition_flow_inv}(ii) shows
that every maximal chain transitive set of a flow is invariant. Thus it is not
necessary to impose these invariance assumptions.
\end{proof}

\subsection{Control sets for linear control systems\label{Subsection2.3}}

Consider linear control systems of the form (\ref{1}), where we now allow
$A\in\mathbb{K}^{n\times n},B\in\mathbb{K}^{n\times m}$, and $U\subset
\mathbb{K}^{m}$.

The decomposition of $\mathbb{K}^{n}$ into the sum of the unstable, center,
and stable subspaces of $A$ is given by%
\begin{equation}
\mathbb{K}^{n}=L^{+}\oplus L^{0}\oplus L^{-}, \label{decomposition1}%
\end{equation}
where $L^{+},L^{0}$, and $L^{-}$ are the Lyapunov spaces for positive,
vanishing, and negative Lyapunov exponents, respectively. Hence $L^{+},L^{0}$,
and $L^{-}$ are the sums of the generalized eigenspaces for eigenvalues with
positive, vanishing, and negative real part, respectively. The center-unstable
and the center-stable subspaces are $L^{+,0}=L^{+}\oplus L^{0}$ and
$L^{-,0}=L^{-}\oplus L^{0}$, respectively. We denote by $\pi^{\pm}%
:\mathbb{K}^{n}\rightarrow L^{\pm}$ the associated projections along
$L^{\mp,0}$ and by $\pi^{h}:\mathbb{K}^{n}\rightarrow L^{+}\oplus L^{-}$ the
projection along $L^{0}$.

\begin{theorem}
\label{Theorem_cs}Assume that $(A,B)$ is controllable, i.e., the reachable
subspace $\mathcal{C}:=\operatorname{Im}[B~AB\cdots A^{n-1}B]$ coincides with
$\mathbb{K}^{n}$.

(i) The controllable set to $0\in\mathbb{K}^{n}$ and the reachable set from
$0\in\mathbb{K}^{n}$ are
\[
\mathbf{C}(0)=(\mathbf{C}(0)\cap L^{+})\oplus L^{-,0}\text{ and }%
\mathbf{R}(0)=L^{+,0}\oplus(\mathbf{R}(0)\cap L^{-}),
\]
respectively, where $\mathbf{C}(0)\cap L^{+}$ is bounded, convex, and open
relative to $L^{+}$ and $\mathbf{R}(0)\cap L^{-}$ is bounded, convex, and open
relative to $L^{-}$.

(ii) There is a unique control set $D$ with nonvoid interior $\mathrm{int}D$.
It is convex and contains the origin $0\in\mathbb{K}^{n}$ in the interior and
$\overline{\mathrm{int}D}=\overline{D}$.

(iii) The control set satisfies%
\[
D=\mathbf{C}(0)\cap\overline{\mathbf{R}(0)}=\left(  \mathbf{C}(0)\cap
L^{+}\right)  \oplus L^{0}\oplus(\overline{\mathbf{R}(0)}\cap L^{-}).
\]

(iv) For all $x\in D$ and $y\in\mathrm{int}D$ there are $t>0$ and
$u\in\mathcal{U}$ with $\varphi(t,x,u)=y$.

(v) For every $x\in\overline{D}$ and $t>0$ there is $u\in\mathcal{U}$ with
$\varphi(t,x,u)\in\overline{D}$.
\end{theorem}

\begin{proof}
(i) The result on $\mathbf{R}(0)$ is proved in Sontag \cite[Corollary
3.6.7]{Son98} and the result for $\mathbf{C}(0)$ follows by time reversal.
Assertion (ii) follows by Colonius and Kliemann \cite[Example 3.2.16]{ColK00}
and Remark \ref{Remark2.2}, which also implies (iv). For assertion (iii) use
decomposition (\ref{decomposition1}) and (i), (ii) to see that%
\[
D=\mathbf{C}(0)\cap\overline{\mathbf{R}(0)}=(\mathbf{C}(0)\cap L^{+})\oplus
L^{0}\oplus(\overline{\mathbf{R}(0)}\cap L^{-}).
\]
For assertion (v) note that by (ii) there are $x_{k}\in\mathrm{int}D$ with
$x_{k}\rightarrow x$, and by (iv) one finds $t_{k}>0,u_{k}\in\mathcal{U}$ with
$\varphi(t_{k},x_{k},u_{k})\in D$. Since the $t_{k}$ may be chosen bounded,
compactness of $\mathcal{U}$ and continuity of $\varphi$ imply that there are
$t>0$ and $u\in\mathcal{U}$ with $\varphi(t,x,u)\in\overline{D}$.
\end{proof}

Next we use these results for general linear systems by applying them to the
restriction to the reachable subspace $\mathcal{C}$.

\begin{corollary}
\label{Corollary_cs}Consider a linear control system of the form (\ref{1}).

(i) There is a control set $D_{0}$ with $0\in D_{0}$ and it satisfies%
\[
D_{0}=\left(  \mathbf{C}(0)\cap L^{+}\right)  \oplus\left(  L^{0}%
\cap\mathcal{C}\right)  \oplus(\overline{\mathbf{R}(0)}\cap L^{-}).
\]
where $\overline{\mathbf{R}(0)}\cap L^{-}\subset\mathcal{C}$ is compact and
$\mathbf{C}(0)\cap L^{+}\subset\mathcal{C}$ is bounded and open in the
relative topology of $L^{+}\cap\mathcal{C}$.

(ii) The origin $0\in\mathbb{K}^{n}$ is in the interior $\mathrm{int}%
_{\mathcal{C}}D_{0}$ of $D_{0}$ relative to $\mathcal{C}$ which is dense in
$D_{0}$, and for all $x\in D_{0}$ and $y\in\mathrm{int}_{\mathcal{C}}D_{0}$
there are $t>0$ and $u\in\mathcal{U}$ with $y=\varphi(t,x,u)$.
\end{corollary}

\begin{proof}
It is clear that there is a unique control set $D_{0}$ containing
$0\in\mathbb{K}^{n}$. The subspace $\mathcal{C}$ is closed and $\varphi
(t,x,u)\in\mathcal{C}$ for all $t\in\mathbb{R},x\in\mathcal{C}$, and
$u\in\mathcal{U}$. By invariance of the subspace $\mathcal{C}$ the control
set, the controllable set, and the reachable set satisfy $D_{0}\subset
\overline{\mathbf{C}(0)}\cup\overline{\mathbf{R}(0)}\subset\mathcal{C}$. Thus
the assertions follow from Theorem \ref{Theorem_cs}.
\end{proof}

We also note the following lemma.

\begin{lemma}
\label{Lemma_projection}(i) The control set $D_{0}^{-}$ containing $0$ of the
system induced on $L^{-}$,%
\begin{equation}
\dot{y}(t)=Ay(t)+\pi^{-}Bu(t),\,u\in\mathcal{U}, \label{P}%
\end{equation}
coincides with the projection $\pi^{-}\mathbf{R}(0)$.

(ii) The control set $D_{0}^{+}$ containing $0$ of the system induced on
$L^{+}$,%
\begin{equation}
\dot{y}(t)=Ay(t)+\pi^{+}Bu(t),\,u\in\mathcal{U}, \label{P+}%
\end{equation}
coincides with the projection $\pi^{+}\mathbf{C}(0)$.

(iii) Consider the system induced on the subspace $L^{+}\oplus L^{-}$%
\begin{equation}
\dot{y}(t)=A^{h}y(t)+\pi^{h}Bu(t),\,u\in\mathcal{U}, \label{hyp}%
\end{equation}
where $A^{h}:=A_{\left\vert L^{+}\oplus L^{-}\right.  }$. Then $A^{h}$ is
hyperbolic, i.e., $\mathrm{spec}(A^{h})\cap\imath\mathbb{R}=\varnothing$. The
reachable set $\mathbf{R}^{h}(0)$, the controllable set $\mathbf{C}^{h}(0)$,
and the control set $D_{0}^{h}$ containing $0$ of (\ref{hyp}) satisfy
$\mathbf{R}^{h}(0)=\pi^{h}\mathbf{R}(0),\mathbf{C}^{h}(0)=\pi^{h}%
\mathbf{C}(0)$ and $D_{0}^{h}=\pi^{h}D_{0}$.
\end{lemma}

\begin{proof}
(i) Since $L^{-}$ and $L^{+,0}$ are $A$-invariant it follows for $x\oplus y$
with $x\in L^{-},y\in L^{+,0}$ that $Ax\in L^{-},Ay\in L^{+,0}$, and hence
$\pi^{-}A(x\oplus y)=Ax=A\pi^{-}(x\oplus y)$. This shows that $\pi^{-}%
A=A\pi^{-}$. For $x\in\pi^{-}\mathbf{R}(0)$ there are $T>0$ and $u\in
\mathcal{U}$ with%
\[
x=\pi^{-}x=\pi^{-}\int_{0}^{T}e^{A(T-s)}Bu(s)ds=\int_{0}^{T}e^{A(T-s)}\pi
^{-}Bu(s)ds\in\mathbf{R}^{-}(0).
\]
Conversely, if $x\in\mathbf{R}^{-}(0)$ there are $T>0$ and $u\in\mathcal{U}$
with%
\[
x=\int_{0}^{T}e^{A(T-s)}\pi^{-}Bu(s)ds=\pi^{-}\int_{0}^{T}e^{A(T-s)}%
Bu(s)ds\in\pi^{-}\mathbf{R}(0).
\]
Since $A_{\left\vert L^{-}\right.  }$ is asymptotically stable it follows from
Corollary \ref{Corollary_cs}(i) that%
\[
D_{0}^{-}=\overline{\mathbf{R}^{-}(0)}=\overline{\pi^{-}\mathbf{R}(0)}=\pi
^{-}\overline{\mathbf{R}(0)}.
\]

(ii) follows analogously.

(iii) Suppose first that $(A,B)$ is controllable. Then it follows that
$0\in\mathrm{int}D_{0}$ and hence $D_{0}=\overline{\mathbf{R}(0)}%
\cap\mathbf{C}(0)$ showing that $0\in\mathrm{int}\mathbf{R}(0)$ and
$0\in\mathrm{int}\mathbf{C}(0)$. Since the projection $\pi^{h}:\mathbb{K}%
^{n}\rightarrow L^{+}\oplus L^{-}$ along $L^{0}$ is open it follows that%
\[
0\in\mathrm{int}\left(  \pi^{h}\mathbf{R}(0)\right)  \text{ and }%
0\in\mathrm{int}\left(  \pi^{h}\mathbf{C}(0)\right)  .
\]
The same arguments as in (i) and (ii) imply that $\mathbf{R}^{h}(0)=\pi
^{h}\mathbf{R}(0)$ and $\mathbf{C}^{h}(0)=\pi^{h}\mathbf{C}(0)$. This shows
that $0\in\mathrm{int}\overline{\mathbf{R}^{h}(0)}$ and $0\in\mathrm{int}%
\mathbf{C}^{h}(0)$ implying
\[
D_{0}^{h}=\overline{\mathbf{R}^{h}(0)}\cap\mathbf{C}^{h}(0)=\overline{\pi
^{h}\mathbf{R}(0)}\cap\pi^{h}\mathbf{C}(0)=\pi^{h}\left(  \overline
{\mathbf{R}(0)}\cap\mathbf{C}(0)\right)  =\pi^{h}D_{0}.
\]
In the general case, the sets $\overline{\mathbf{R}(0)},\mathbf{C}(0)$, and
the control set $D_{0}$ containing $0$ are contained in the controllable
subspace $\mathcal{C}$ and $D_{0}$ has nonvoid interior relative to
$\mathcal{C}$. Observe also that $\ker\pi^{h}=L^{0}$, hence $D_{0}\subset
D_{0}^{h}+L^{0}$ and similarly for $\mathbf{R}(0)$ and $\mathbf{C}(0)$.
\end{proof}

\section{Chain transitivity for autonomous linear differential
equations\label{Section3}}

We prove that the chain transitive set of the flow for autonomous linear
differential equations coincides with the center Lyapunov space. This also
yields first results on chain controllability properties of the linear control
system (\ref{1}).

The following lemma contains the crux of the argument. Let $\psi
:\mathbb{R}\times\mathbb{K}^{n}\rightarrow\mathbb{K}^{n},\psi(t,x)=e^{At}%
x,t\in\mathbb{R},x\in\mathbb{K}^{n}$, be the flow of the differential equation
$\dot{x}=Ax$ for $A\in\mathbb{K}^{n\times n}$.

\begin{lemma}
\label{Lemma2_Eduardo}If all eigenvalues of the matrix $A$ have null real
parts, then for all $x,y\in\mathbb{K}^{n}$ and all $\varepsilon,T>0$ there is
an $(\varepsilon,T)$-chain of the flow $\psi$ from $x$ to $y$.
\end{lemma}

\begin{proof}
We first prove this lemma for the complex case $\mathbb{K}=\mathbb{C}$ using
induction over the dimension $n$. The lemma holds trivially for $n=0$ since
$\mathbb{C}^{0}$ consists of a single point. Assume that the assertion holds
for all $A\in\mathbb{C}^{(n-1)\times(n-1)}$. Since $\mathbb{C}$ is
algebraically closed, $A$ has some eigenvector $v$ associated to an eigenvalue
$\mu\in\mathbb{C}$ and, by hypothesis, $\operatorname{Re}\mu=0$. There is
$r>0$ such that $r\mu=2j\pi\imath$ for some $j\in\mathbb{Z}$. It follows that
for all $\lambda\in\mathbb{C}$%
\[
e^{rA}\lambda v=\lambda e^{2j\pi\imath}v=\lambda v.
\]
Thus the restriction of $\psi$ to the linear span $\left\langle v\right\rangle
$ of $v$ is periodic. We claim that for all $\lambda_{1}v,\lambda_{2}%
v\in\left\langle v\right\rangle ,\lambda_{1},\lambda_{2}\in\mathbb{C}$, and
all $\varepsilon,T>0$ there is an $(\varepsilon,T)$-chain from $\lambda_{1}v$
to $\lambda_{2}v$: In fact, consider periodic solutions through $\lambda
_{1}v+\frac{i}{j}(\lambda_{2}-\lambda_{1})v,i\in\{0,1,\ldots,j\}$, where
$j\in\mathbb{N}$ is large enough such that%
\[
\left\Vert \lambda_{1}v+\frac{i}{j}(\lambda_{2}-\lambda_{1})v-\left[
\lambda_{1}v+\frac{i+1}{j}(\lambda_{2}-\lambda_{1})v\right]  \right\Vert
=\frac{1}{j}\left\vert \lambda_{2}-\lambda_{1}\right\vert \left\Vert
v\right\Vert <\varepsilon.
\]
Staying on each periodic solution long enough such that the time is greater
than $T$ and jumping from one periodic solution to the next yields an
$(\varepsilon,T)$-chain from $\lambda_{1}v$ to $\lambda_{2}v$.

Now let $W\subset\mathbb{C}^{n}$ be a subspace such that $\mathbb{C}%
^{n}=W\oplus\left\langle v\right\rangle $. There are linear transformations
$A_{1}:W\rightarrow W$ and $A_{2}:W\rightarrow\left\langle v\right\rangle $
such that $Aw=A_{1}w+A_{2}w$. We obtain for $w\in W$,%
\[
\frac{d}{dt}\psi(t,w)=A\psi(t,w)=A_{1}\psi(t,w)+A_{2}\psi(t,w)\text{ with
}A_{2}\psi(t,w)\in\langle v\rangle.
\]
The variation-of-parameters formula for $A_{1}$ shows that%
\[
\psi(t,w)=e^{At}w=e^{tA_{1}}w+\widehat{\psi}(t,w)\text{ with }\widehat{\psi
}(t,w):=\int_{0}^{t}e^{(t-s)A_{1}}A_{2}\psi(s,w)ds.
\]
On the other hand, $\psi(t,\lambda v)=e^{At}\lambda v=e^{t\mu}\lambda v$ and
hence by linearity%
\begin{equation}
\psi(t,w+\lambda v)=e^{tA_{1}}w+\widehat{\psi}(t,w)+e^{t\mu}\lambda v.
\label{e_1}%
\end{equation}

Let $x\in\mathbb{C}^{n}$. We will show that $0\in\Omega(x)$ and $x\in
\Omega(0)$. This implies the assertion, since any two points $x,y$ can be
connected by concatenating a chain from $x$ to $0$ with a chain from $0$ to
$y$.

Let $\varepsilon,T>0$ and write $x=w\oplus\lambda v\in\mathbb{C}^{n}$. By the
induction hypothesis for all $\varepsilon,T>0$ there is an $(\varepsilon
,T)$-chain from $w$ to $0$ given by $w_{0}=w,w_{1},\ldots,\allowbreak w_{k}=0$
and $T_{0},T_{1},\ldots,T_{k-1}\geq T$ with%
\[
\left\Vert e^{T_{i}A_{1}}w_{i}-w_{i+1}\right\Vert <\varepsilon\text{ for
}i=0,\ldots,k-1.
\]
Next we construct an $(\varepsilon,T)$-chain from $x=w\oplus\lambda
v\in\mathbb{C}^{n}$ to $0$. Define recursively%
\[
v_{0}=\lambda v,v_{i+1}=\widehat{\psi}(T_{i},w_{i})+e^{T_{i}\mu}v_{i}%
\in\left\langle v\right\rangle \text{ for }i=1,\ldots,k-1.
\]
Hence by (\ref{e_1})%
\begin{align}
&  e^{T_{i}A}(w_{i}+v_{i})-(w_{i+1}+v_{i+1})\label{e_2}\\
&  =e^{T_{i}A_{1}}w_{i}+\widehat{\psi}(T_{i},w_{i})+e^{T_{i}\mu}v_{i}%
-w_{i+1}-\widehat{\psi}(T_{i},w_{i})-e^{T_{i}\mu}v_{i}=e^{T_{i}A_{1}}%
w_{i}-w_{i+1}.\nonumber
\end{align}
For $x_{i}:=w_{i}+v_{i},i=0,\ldots,k$, it follows that%
\[
\left\Vert e^{T_{i}A}x_{i}-x_{i+1}\right\Vert =\left\Vert e^{T_{i}A_{1}}%
w_{i}-w_{i+1}\right\Vert <\varepsilon.
\]
Thus $x_{0}=x,x_{1},\ldots,x_{k}=v_{k}$ and $T_{0},T_{1},\ldots,T_{k-1}\geq T$
is an $(\varepsilon,T)$-chain from $x$ to $v_{k}\in\left\langle v\right\rangle
$. Concatenating this with an $(\varepsilon,T)$-chain from $v_{k}$ to $0$ in
$\left\langle v\right\rangle $ we obtain the desired chain.

The inclusion $x\in\Omega(0)$ follows from considering the time reversed
system and Proposition \ref{LemCLSAtReverseTime}.

Now assume $\mathbb{K}=\mathbb{R}$. The previous result yields that for all
$x,y\in\mathbb{R}^{n}\subset\mathbb{C}^{n}$ and for all $\varepsilon,T>0$
there are $(\varepsilon,T)$-chains in $\mathbb{C}^{n}$ from $x$ to $y$.
Consider such a chain given by $z_{0}=x,z_{1},\ldots,z_{k}=y$ in
$\mathbb{C}^{n}$ and $T_{0},\ldots,T_{k-1}\geq T$ with%
\[
\left\Vert e^{T_{i}A}z_{i}-z_{i+1}\right\Vert <\varepsilon\text{ for
}i=0,\ldots,k-1.
\]
Since the entries of $A$ are real, it follows that for all $i$%
\[
\left\Vert e^{T_{i}A}\operatorname{Re}z_{i}-\operatorname{Re}z_{i+1}%
\right\Vert =\left\Vert \operatorname{Re}\left(  e^{T_{i}A}z_{i}%
-z_{i+1}\right)  \right\Vert \leq\left\Vert e^{T_{i}A}z_{i}-z_{i+1}\right\Vert
<\varepsilon.
\]
This shows that one obtains an $(\varepsilon,T)$-chain on $\mathbb{R}^{n}$
from $z_{0}=x$ to $z_{k}=y$.
\end{proof}

A consequence of this lemma is the following theorem characterizing the chain
recurrent set of linear autonomous differential equations.

\begin{theorem}
\label{Theorem_1_Eduardo}For a differential equation $\dot{x}(t)=Ax(t)$ with
$A\in\mathbb{K}^{n\times n}$ the chain recurrent set of its flow $\psi(t,x)$
coincides with the center Lyapunov space $L^{0}$.
\end{theorem}

\begin{proof}
By Lemma \ref{Lemma2_Eduardo} the set $L^{0}$ is chain transitive. It remains
to show that any chain recurrent point is in $L^{0}$. Since all norms in
$\mathbb{K}^{n}$ are equivalent we can assume without loss in generality that
$L^{+}\oplus L^{-}=\left(  L^{0}\right)  ^{\bot}$ which implies%
\begin{equation}
\left\Vert v\right\Vert \leq\left\Vert v+w\right\Vert \text{ for all }v\in
L^{+}\oplus L^{-},w\in L^{0}. \label{norm}%
\end{equation}
Let $x\in\mathbb{K}^{n}$ be a chain recurrent point. Thus for all
$\varepsilon,T>0$ there is an $(\varepsilon,T)$-chain $x_{0}=x,x_{1}%
,\ldots,x_{k}=x$ with $T_{0},T_{1},\ldots,T_{k-1}\geq T$. Write%
\[
x=v\oplus w\text{ and }x_{i}=v_{i}\oplus w_{i}\text{ with }v,v_{i}\in
L^{+}\oplus L^{-}\text{ and }w,w_{i}\in L^{0}.
\]
Then it follows from (\ref{norm}) that for all $i$%
\[
\left\Vert e^{T_{i}A}v_{i}-v_{i+1}\right\Vert \leq\left\Vert e^{T_{i}A}%
v_{i}+e^{T_{i}A}w_{i}-v_{i+1}-w_{i+1}\right\Vert =\left\Vert e^{T_{i}A}%
x_{i}-x_{i+1}\right\Vert <\varepsilon.
\]
This shows that $v_{0}=v,v_{1},\ldots,v_{k}=v$ is an $(\varepsilon,T)$-chain
from $v$ to $v$ for the flow restricted to $L^{+}\oplus L^{-}$. Since
$\varepsilon,T>0$ are arbitrary and this flow is hyperbolic, it follows that
$v=0$. In fact, Antunez, Mantovani, and Var\~{a}o \cite[Corollary
2.11]{AntMV22} shows this for discrete linear flows and similar arguments can
be used for continuous linear flows.
\end{proof}

Recall from Subsection \ref{Subsection2.1} the definition of chain limits set
$\Omega(x)$ for flows.

\begin{proposition}
\label{Proposition_A}For the flow $\psi$ on $\mathbb{K}^{n}$ the
center-unstable subspace and the center-stable subspace satisfy $L^{+,0}%
\subset\Omega(0)$ and $L^{-,0}\subset\Omega^{\ast}(0)$, respectively.
\end{proposition}

\begin{proof}
Let $x=x_{+}\oplus x_{0}\in L^{+,0}$ be arbitrary, where $x_{+}\in L^{+}$ and
$x_{0}\in L^{0}$, and fix $\varepsilon,T>0$. For $S\geq T$ sufficiently large,
$\Vert e^{-SA}x_{+}\Vert<\varepsilon$. Since $L^{0}$ is $A$-invariant it
follows that $e^{-2SA}x_{0}\in L^{0}$. Hence, by chain transitivity of $L^{0}%
$, there is an $(\varepsilon,T)$ chain $0=y_{0},y_{1},\ldots,y_{k}%
=e^{-2SA}x_{0}$ with times $T_{0},T_{1},\ldots,T_{k-1}\geq T$. Define an
$(\varepsilon,T)$-chain from $0$ to $x$ by $x_{0}=y_{0}=0,x_{1}=y_{1}%
,\ldots,x_{k}=y_{k}=e^{-2SA}x_{0},x_{k+1}=e^{-SA}\left(  x_{+}+x_{0}\right)
,x_{k+2}=x_{+}+x_{0}=x$ with times $T_{0},\ldots,T_{k-1},T_{k}=T_{k+1}=S$.
This holds since%
\begin{align*}
\left\Vert \psi(T_{i},x_{i})-x_{i+1}\right\Vert  &  =\left\Vert \psi
(T_{i},y_{i})-y_{i+1}\right\Vert <\varepsilon\text{ for }i=0,\ldots,k-1,\\
\left\Vert \psi(T_{k},x_{k})-x_{k+1}\right\Vert  &  =\left\Vert e^{SA}%
e^{-2SA}x_{0}-\left(  e^{-SA}x_{+}+e^{-SA}x_{0}\right)  \right\Vert
=\left\Vert e^{-SA}x_{+}\right\Vert <\varepsilon,\\
\left\Vert \psi(T_{k+1},x_{k+1})-x_{k+2}\right\Vert  &  =\left\Vert
e^{SA}\left(  e^{-SA}x_{+}+e^{-SA}x_{0}\right)  -\left(  x_{+}+x_{0}\right)
\right\Vert =0.
\end{align*}
The assertion for $L^{-,0}$ follows by considering the time reversed system.
\end{proof}

The following simple example illustrates that for $T\rightarrow\infty$ the
considered chains may become unbounded. Hence in the noncompact state space
$\mathbb{R}^{n}$ it is not sufficient to consider chain transitivity
restricted to compact subsets of $\mathbb{R}^{n}$.

\begin{example}
For the differential equation%
\[
\left(
\begin{array}
[c]{c}%
\dot{x}(t)\\
\dot{y}(t)
\end{array}
\right)  =\left(
\begin{array}
[c]{cc}%
0 & 1\\
0 & 0
\end{array}
\right)  \left(
\begin{array}
[c]{c}%
x(t)\\
y(t)
\end{array}
\right)
\]
the $x$-axis consists of equilibria. In the open upper half-plane all
trajectories move to the right on parallels to the $x$-axis, and in the lower
half-plane they move to the left. The state space $\mathbb{R}^{2}$ is chain
transitive, and for $T\rightarrow\infty$ the $(\varepsilon,T)$-chains become unbounded.
\end{example}

We return to control system (\ref{1}). The following proposition relates the
control set $D_{0}$ containing $0$ and the chain reachable and chain
controllable sets to the spectral subspaces of the homogeneous part $\dot
{x}=Ax$.

\begin{proposition}
\label{Proposition1_2_Eduardo}Consider a linear control system of the form
(\ref{1}).

(i) The inclusions $\overline{D_{0}}+L^{-,0}\subset\mathbf{C}^{c}(0)$ and
$\overline{D_{0}}+L^{+,0}\subset\mathbf{R}^{c}(0)$ hold.

(ii) The set $\overline{D_{0}}+L^{0}$ is chain controllable.
\end{proposition}

\begin{proof}
(i) First we prove the following claim:%
\[
\mathbf{R}(z)\cap L^{-,0}\not =\varnothing\text{ for }z\in\mathrm{int}%
_{\mathcal{C}}D_{0}+L^{-,0}\text{ and }\mathbf{C}(z)\cap L^{+,0}%
\not =\varnothing\text{ for }z\in\mathrm{int}_{\mathcal{C}}D_{0}+L^{+,0}.
\]
Write $z=x+y$ with $x\in\mathrm{int}_{\mathcal{C}}D_{0}$ and $y\in L^{\pm,0}$.
By linearity it follows that%
\[
\varphi(t,z,u)=\varphi(t,x,u)+\varphi(t,y,0)\text{ for all }t\in
\mathbb{R}\text{ and }u\in\mathcal{U}.
\]
Since $0,x\in\mathrm{int}_{\mathcal{C}}D_{0}$ Corollary \ref{Corollary_cs}(ii)
implies that there are $T_{1},T_{2}\geq0$ and $u_{1},u_{2}\in\mathcal{U}$ with
$\varphi(T_{1},x,u_{1})=0$ and $\varphi(-T_{2},x,u_{2})=0$. Using that
$L^{\pm,0}$ is invariant for the control $u\equiv0$ one finds%
\begin{align*}
\varphi(T_{1},z,u_{1})  &  =\varphi(T_{1},y,0)\in L^{-,0}\text{ if }%
z\in\mathrm{int}_{\mathcal{C}}D_{0}+L^{-,0},\\
\varphi(-T_{2},z,u_{2})  &  =\varphi(-T_{2},y,0)\in L^{+,0}\text{ if }%
z\in\mathrm{int}_{\mathcal{C}}D_{0}+L^{+,0}.
\end{align*}
This implies the claim. Observe that we can take the times $T_{1},T_{2}$
arbitrarily large, when we extend $u_{1}$ and $u_{2}$ by the control
$u\equiv0$.

Next we show that $\overline{D_{0}}+$ $L^{-,0}\subset\mathbf{C}^{c}(0)$. Since
$\mathcal{C}$ is a closed subspace and the set $D_{0}$ is contained in the
closure of $\mathrm{int}_{\mathcal{A}}D_{0}$, it follows that $\overline
{D_{0}}=\overline{\mathrm{int}_{\mathcal{A}}D_{0}}$. Recalling that
$\mathbf{R}^{c}(0)$ is closed by Proposition \ref{Proposition_ccs1}(iii) one
sees that it suffices to prove that $\mathrm{int}_{\mathcal{A}}D_{0}%
+L^{-,0}\subset\mathbf{C}^{c}(x)$. Let $z\in\mathrm{int}_{\mathcal{A}}%
D_{0}+L^{-,0}$ and fix $\varepsilon,T>0$. By the claim there is a point $y\in
L^{-,0}$ which can be reached from $z$ at a time greater than $T$. By
Proposition \ref{Proposition_A}\textbf{ }there is a controlled $(\varepsilon
,T)$-chain (with control $u\equiv0$) from $y\in L^{-,0}$ to $0$. Concatenating
this with the trajectory segment from $z$ to $y$ one obtains a controlled
$(\varepsilon,T)$-chain from $z$ to $0$.

For the inclusion $\overline{D_{0}}+L^{+,0}\subset\mathbf{R}^{c}(0)$ it
similarly suffices to show that $\mathrm{int}_{\mathcal{A}}D_{0}+L^{+,0}$ is
contained in the closed set $\mathbf{R}^{c}(0)$. Let $z\in\mathrm{int}%
_{\mathcal{A}}D_{0}+L^{+,0}$ and fix $\varepsilon,T>0$. By the claim there is
a point $y\in L^{+,0}$ such that $z$ can be reached from $y$ at a time greater
than $T$. By Proposition \ref{Proposition_A} there is a controlled
$(\varepsilon,T)$-chain from $0$ to $y$. Concatenation of these controlled
chains yields a controlled $(\varepsilon,T)$-chain from $0$ to $z$.

(ii) This follows from assertion (i) since%
\[
\overline{D_{0}}+L^{0}=(\overline{D_{0}}+L^{+,0})\cap(\overline{D_{0}}%
+L^{-,0})\subset\mathbf{R}^{c}(0)\cap\mathbf{C}^{c}(0).
\]

\end{proof}

\section{The chain control set in $\mathbb{R}^{n}$\label{Section4}}

In this section we first recall results from Colonius and Santana
\cite{ColS23} on the projection to the Poincar\'{e} sphere. Together with
attractor-repeller decompositions this will enable us to prove that the chain
controllable set $\overline{D_{0}}+L_{0}$ in Proposition
\ref{Proposition1_2_Eduardo} actually coincides with the chain control set.

Linear control systems of the form (\ref{1}) on $\mathbb{R}^{n}$ can be lifted
to bilinear control systems with states $(x(t),y(t))$ in $\mathbb{R}^{n}%
\times\mathbb{R}=\mathbb{R}^{n+1}$ by%
\begin{equation}
\dot{x}(t)=Ax(t)+y(t)Bu(t),~\dot{y}(t)=0,\text{ }u\in\mathcal{U}.
\label{lift1}%
\end{equation}
The solutions for initial condition $\left(  x(0),y(0)\right)  =\left(
x_{0},r\right)  \in\mathbb{R}^{n}\times\mathbb{R}$, may be written as
\[
\varphi^{1}(t,x_{0},r,u)=\left(  e^{At}x_{0}+r\int_{0}^{t}e^{A(t-s)}%
Bu(s)ds,r\right)  ,\,t\geq0.
\]
Observe that for $r=0$ one has $\varphi^{1}(t,x,0,u)=\left(  \psi
(t,x),0\right)  $ and for $r=1$ one has $\varphi^{1}(t,x,1,u)=(\varphi
(t,x,u),1)$ for all $t\in\mathbb{R},x\in\mathbb{R}^{n},u\in\mathcal{U}$. Hence
on the hyperplane $\mathbb{R}^{n}\times\{1\}\subset\mathbb{R}^{n+1}$ one
obtains a copy of control system (\ref{1}).

Define subsets of $\mathbb{R}^{n+1}$, of projective space $\mathbb{P}^{n}$,
and of the unit sphere $\mathbb{S}^{n}$ by%
\begin{align}
\mathbb{R}^{n+1,0}  &  =\mathbb{R}^{n}\times\{0\},\,\mathbb{R}^{n+1,1}%
=\mathbb{R}^{n}\times(\mathbb{R}\setminus\{0\}),\nonumber\\
\mathbb{P}^{n,0}  &  =\{\mathbb{P}(x,0)\left\vert x\in\mathbb{R}^{n}\right.
\},\,\mathbb{P}^{n,1}=\{\mathbb{P}(x,r)\left\vert x\in\mathbb{R}^{n}%
,r\not =0\right.  \},\label{4.2}\\
\mathbb{S}^{n,+}  &  :=\left\{  (x,r)\in\mathbb{S}^{n}\left\vert
x\in\mathbb{R}^{n},r>0\right.  \right\}  ,\,\mathbb{S}^{n,0}=\{(x,0)\in
\mathbb{S}^{n}\left\vert x\in\mathbb{R}^{n}\right.  \},\nonumber
\end{align}
respectively. Observe that $\mathbb{P}^{n,1}=\{\mathbb{P}(x,1)\left\vert
x\in\mathbb{R}^{n}\right.  \}$ can be identified with the \textquotedblleft
northern\textquotedblright\ hemisphere $\mathbb{S}^{n,+}$ of $\mathbb{S}^{n}$
and $\mathbb{P}^{n,0}$ is the projection of the equator $\mathbb{S}^{n,0}$.

\begin{definition}
The projective Poincar\'{e} sphere is $\mathbb{P}^{n}$ and the projective
Poincar\'{e} bundle is $\mathcal{U}\times\mathbb{P}^{n}$.
\end{definition}

Control system (\ref{lift1}) induces a control flow $\Phi^{1}$ on
$\mathcal{U}\times\mathbb{R}^{n+1}$ defined by%
\begin{equation}
\Phi_{t}^{1}(u,x,r)=(u(t+\cdot),\varphi^{1}(t,x,r,u)),t\in\mathbb{R}%
,(x,r)\mathbb{\in R}^{n+1},\,u\in\mathcal{U}. \label{Fi_1}%
\end{equation}
The maps between the fibers $\{u\}\times\mathbb{R}^{n+1}$ are linear, hence
$\Phi^{1}$ is a linear flow: For every $u\in\mathcal{U}$ and $\alpha,\beta
\in\mathbb{R},(x_{1},r_{1}),(x_{2},r_{2})\in\mathbb{R}^{n}\times\mathbb{R}$
\[
\alpha\varphi^{1}(t,x_{1},r_{1},u)+\beta\varphi^{1}(t,x_{2},r_{2}%
,u)=\varphi^{1}(t,\alpha x_{1}+\beta x_{2},\alpha r_{1}+\beta r_{2},u).
\]
The projection of system (\ref{lift1}) to projective space yields a projective
control flow $\mathbb{P}\Phi^{1}$ on the projective Poincar\'{e} bundle
$\mathcal{U}\times\mathbb{P}^{n}$. The subsets $\mathcal{U}\times
\mathbb{P}^{n,0}$ and $\mathcal{U}\times\mathbb{P}^{n,1}$ are invariant under
the flow $\mathbb{P}\Phi^{1}$. The following proposition shows some relations
between the control flows on $\mathcal{U}\times\mathbb{R}^{n}$ and on the
projective Poincar\'{e} bundle.

\begin{proposition}
\label{Proposition_e}(i) Every $(u,x)\in\mathcal{U}\times\mathbb{R}^{n}$
satisfies the equality $\Phi_{t}^{1}(u,x,0)=(u(t+\cdot),\varphi^{1}%
(t,x,0,u))=(u(t+\cdot),\psi(t,x),0),\allowbreak t\in\mathbb{R}$, and the
projective map%
\begin{equation}
h^{0}:\mathcal{U}\times\mathbb{P}^{n-1}\rightarrow\mathcal{U}\times
\mathbb{P}^{n,0},h^{0}(u,\mathbb{P}x)=(u,\mathbb{P}(x,0)), \label{h0}%
\end{equation}
is a conjugacy of the flow on $\mathcal{U}\times\mathbb{P}^{n-1}$ given by
$(u,p)\mapsto(u(t+\cdot),\mathbb{P}\psi(t,p))$ and the flow $\mathbb{P}%
\Phi^{1}$ restricted to $\mathcal{U}\times\mathbb{P}^{n,0}$.

(ii) The map%
\begin{equation}
h^{1}:\mathcal{U}\times\mathbb{R}^{n}\rightarrow\mathcal{U}\times
\mathbb{P}^{n,1},(u,x)\mapsto(u,\mathbb{P}(x,1)), \label{h1}%
\end{equation}
is a conjugacy of the flows $\Phi$ on $\mathcal{U}\times\mathbb{R}^{n}$ and
$\mathbb{P}\Phi^{1}$ restricted to $\mathcal{U}\times\mathbb{P}^{n,1}$.

(iii) For $\varepsilon,T>0$ any $(\varepsilon,T)$-chain in $\mathcal{U}%
\times\mathbb{R}^{n}$ is mapped by $h^{1}$ onto a $(2\varepsilon,T)$-chain in
$\mathcal{U}\times\mathbb{P}^{n,1}$, hence any chain transitive set
$C\subset\mathcal{U}\times\mathbb{R}^{n}$ is mapped onto a chain transitive
set $h^{1}(C)\subset\mathcal{U}\times\mathbb{P}^{n,1}$.

(iv) For a subset $C\subset\mathcal{U}\times\mathbb{R}^{n}$ the set
$\{x\in\mathbb{R}^{n}\left\vert (u,x)\in C\text{ for some }u\in\mathcal{U}%
\right.  \}$ is bounded if and only if $\overline{h^{1}(C)}\cap(\mathcal{U}%
\times\mathbb{P}^{n,0})=\varnothing$.
\end{proposition}

\begin{proof}
Assertion (i) holds, since $\varphi^{1}(t,x,0,u)=(\psi(t,x),0),t\in
\mathbb{R},x\in\mathbb{R}^{n},u\in\mathcal{U}$, implying (\ref{h0}) for the
corresponding projective flows. Assertions (ii)-(iv) are a special case of
Colonius and Santana \cite[Proposition 9]{ColS23}.
\end{proof}

Applying Selgrade's theorem, Theorem \ref{Theorem_Selgrade}, to lifted control
flows $\Phi^{1}$ for general split affine control systems one obtains the
following theorem; cf. \cite[Corollary 32, Remark 34]{ColS23}. Below we will
sharpen these results for linear control systems.

\begin{theorem}
\label{Theorem_Selgrade1}Consider a linear control system of the form
(\ref{1}) on $\mathbb{R}^{n}$ and its lift (\ref{lift1}) to $\mathbb{R}^{n+1}$.

(i) The linear flow $(\theta_{t}u,\psi(t,x))$ on $\mathcal{U}\times
\mathbb{R}^{n}$ defined by the linear part $\dot{x}=Ax$ of (\ref{1}) has the
Selgrade decomposition%
\[
\mathcal{U}\times\mathbb{R}^{n}=\bigoplus\nolimits_{i}\mathcal{V}_{i}\text{
with }\mathcal{V}_{i}=\mathcal{U}\times L(\lambda_{i}),
\]
where $L(\lambda_{i})$ are the Lyapunov spaces of $A$.

(ii) For the lifted flow $\Phi^{1}$ in (\ref{Fi_1}) the Selgrade decomposition
has the form
\begin{equation}
\mathcal{U}\times\mathbb{R}^{n+1}=\mathcal{V}_{1}^{\infty}\oplus\cdots
\oplus\mathcal{V}_{\ell^{+}}^{\infty}\oplus\mathcal{V}_{c}^{1}\oplus
\mathcal{V}_{\ell^{+}+\ell^{0}+1}^{\infty}\oplus\cdots\oplus\mathcal{V}_{\ell
}^{\infty}, \label{Sel4}%
\end{equation}
for some numbers $\ell^{+},\ell^{0}\geq0$ with $\ell^{+}+\ell^{0}\leq\ell$
and
\[
\mathcal{V}_{i}^{\infty}:=\mathcal{V}_{i}\times0=\mathcal{U}\times
L(\lambda_{i})\times0\subset\mathcal{U}\times\mathbb{R}^{n+1,0},i\in
\{1,\ldots,\ell^{+}\}\cup\{\ell^{+}+\ell^{0}+1,\ldots,\ell\}.
\]

(iii) The central Selgrade bundle $\mathcal{V}_{c}^{1}$ satisfies
\[
\mathcal{V}_{c}^{1}\cap\left(  \mathcal{U}\times\mathbb{R}^{n}\times0\right)
=\bigoplus_{i=\ell^{+}+1}^{i=\ell^{+}+\ell^{0}}\mathcal{V}_{i}^{\infty
}:=\mathcal{V}_{c}^{\infty}\text{ and }\dim\mathcal{V}_{c}^{1}=1+\dim
\mathcal{V}_{c}^{\infty}.
\]
The indices $i\in\{\ell^{+}+1,\ldots,\ell^{+}+\ell^{0}\}$ are the indices such
that $h^{1}(\mathcal{U}\times L(\lambda_{i}))=\mathcal{U}\times\mathbb{P}%
\left(  L(\lambda_{i})\times\left\{  1\right\}  \right)  \subset
\mathcal{U}\times\mathbb{P}^{n}$ is a chain transitive set.

(iv) If $A$ is hyperbolic, i.e., $\mathrm{spec}A\cap\imath\mathbb{R}%
=\varnothing$, the central Selgrade bundle is the line bundle%
\[
\mathcal{V}_{c}^{1}=\left\{  (u,-re(u,0),r)\in\mathcal{U}\times\mathbb{R}%
^{n}\times\mathbb{R}\left\vert u\in\mathcal{U},r\in\mathbb{R}\right.
\right\}  ,
\]
where $e(u,t),t\in\mathbb{R}$, is the unique bounded solution of (\ref{1}) for
$u\in\mathcal{U}$, and the projection $\mathcal{M}_{c}^{1}=\mathbb{P}%
\mathcal{V}_{c}^{1}$ satisfies $\mathcal{M}_{c}^{1}\subset\mathbb{P}^{n,1}$.
\end{theorem}

Theorem \ref{Theorem_Selgrade1} implies the following consequences for chain
control sets (cf. \cite[Theorem 35]{ColS23}).

\begin{corollary}
\label{Corollary_splitE}Consider a linear control system of the form (\ref{1})
on $\mathbb{R}^{n}$.

(i) For the induced control system on the projective Poincar\'{e} sphere
$\mathbb{P}^{n}$ there exists a unique chain control set $E_{c}^{1}$ with
$E_{c}^{1}\cap\mathbb{P}^{n,1}\not =\varnothing$. It is given by%
\[
E_{c}^{1}=\{\mathbb{P}x\in\mathbb{P}^{n}\left\vert \exists u\in\mathcal{U}%
:(u,\mathbb{P}x)\in\mathcal{M}_{c}^{1}\right.  \}.
\]
Furthermore, $E_{c}^{1}$ contains $\mathbb{P}\left(  L(0)\times0\right)  $ and
the image~$\mathbb{P}\left(  0_{n},1\right)  $ of the north pole $(0_{n},1)$.

(ii) For the unique chain control set $E$ in $\mathbb{R}^{n}$ of the linear
control system (\ref{1}), the image $\mathbb{P}\left(  E\times\left\{
1\right\}  \right)  $ in the projective Poincar\'{e} sphere $\mathbb{P}^{n}$
is contained in $E_{c}^{1}$.

(iii) If $A$ is hyperbolic, then the chain control set $E_{c}^{1}$ equals
$\mathbb{P}(E\times\left\{  1\right\}  )$, it is a compact subset of
$\mathbb{P}^{n,1}$, and for every $u\in\mathcal{U}$ there is a unique element
$x\in E$ with $\varphi(t,x,u)\in E$ for all $t\in\mathbb{R}$.
\end{corollary}

The next proposition identifies a subbundle occurring in assertion (iii) of
Theorem \ref{Theorem_Selgrade1} (in Theorem \ref{Theorem_central}\textbf{ }we
will show that it is the only one). It is the bundle for the center Lyapunov
space of $A$ and it is convenient to denote it by $\mathcal{V}_{0}%
:=\mathcal{U}\times L(0).$

\begin{proposition}
\label{Proposition_in}The inclusion $\mathcal{V}_{0}^{\infty}:=\mathcal{U}%
\mathbf{\times}L(0)\times0\subset\mathcal{V}_{c}^{1}$ holds$.$
\end{proposition}

\begin{proof}
Theorem \ref{Theorem_1_Eduardo} shows that the subspace $L^{0}=L(0)$ is chain
transitive for the flow $\psi$. This implies by Proposition
\ref{Proposition_e}(iii) that $\mathbb{P}\left(  L^{0}\times\left\{
1\right\}  \right)  $ is chain transitive for the projectivized flow
$\mathbb{P}\psi$ on the compact space $\mathbb{P}^{n}$. The set $\mathcal{U}$
is also compact and chain transitive for the shift flow $\theta$. By Alongi
and Nelson \cite[Theorem 2.7.18]{AlonN07} for both flows it suffices to
consider chains with all jump times equal to $1$. This implies that
$h^{1}(\mathcal{U}\times L(0))=\mathcal{U}\times\mathbb{P}\left(
L(0)\times\left\{  1\right\}  \right)  $ is chain transitive, since the
restriction of the flow $\mathbb{P}\Phi^{1}$ is a product flow noting that the
flow on $\mathbb{P}\left(  L(0)\times\left\{  1\right\}  \right)  $ does not
depend on the element in $\mathcal{U}$. It follows that also $\mathcal{U}%
\times\mathbb{P}\left(  L(0)\times0\right)  $ is chain transitive and hence
$\mathbb{P}\mathcal{V}_{0}^{\infty}\subset\mathbb{P}\mathcal{V}_{c}^{1}$.
\end{proof}

In order to determine the chain control set $E$ we prepare the following results.

\begin{proposition}
\label{PropositionL-}Assume that $\mathbb{R}^{n}=L^{-}$, hence the homogeneous
part $\dot{x}=Ax$ is asymptotically stable. Then the following assertions hold.

(i) The control set $D_{0}$ containing \thinspace$0\in\mathbb{R}^{n}$ is
compact and given by $D_{0}=\overline{\mathbf{R}(0)}$.

(ii) The image $\mathbb{P}(D_{0}\times\left\{  1\right\}  )$ on the projective
Poincar\'{e} sphere is a compact control set contained in $\mathbb{P}^{n,1}$
for the induced system on $\mathbb{P}^{n}$

(iii) The set $\mathbb{P}\left(  D_{0}\times\left\{  1\right\}  \right)  $ is
chain controllable for the induced system on $\mathbb{P}^{n}$.

(iv) The set $\mathcal{A}\subset\mathcal{U}\times\mathbb{P}^{d,1}$ defined by%
\begin{equation}
\mathcal{A}:=\left\{  (u,\mathbb{P}(x,1))\in\mathcal{U}\times\mathbb{P}%
^{d}\left\vert \mathbb{P}(\varphi(t,x,u),1)\in\mathbb{P}(D_{0}\times\left\{
1\right\}  )\text{ for all }t\in\mathbb{R}\right.  \right\}  \label{lift2}%
\end{equation}
is chain transitive for the control flow $\mathbb{P}\Phi^{1}$.

(v) $\mathcal{A}$ is an attractor and its complementary repeller is the set
$\mathcal{A}^{\ast}=\mathcal{U}\times\mathbb{P}^{d,0}$.

(vi) The set $\mathbb{P}(D_{0}\times\left\{  1\right\}  )\subset
\mathbb{P}^{n,1}$ is the unique chain control set in $\mathbb{P}^{n}$.
\end{proposition}

\begin{proof}
(i) This follows by the characterization of the control set $D_{0}$ in
Corollary \ref{Corollary_cs}.

(ii) This holds since by Proposition \ref{Proposition_e}(ii) the map $h^{1}$
is a conjugacy and $D_{0}$ is compact.

(iii) The set $D_{0}$ is a control set with nonvoid interior relative to the
controllability subspace $\mathcal{C}$. Since $h^{1}$ is a conjugacy, it
follows that $\mathbb{P}(D_{0}\times\left\{  1\right\}  )$ is a control set
with nonvoid interior for the system on $\mathbb{P}^{n}$ restricted to
$\mathbb{P}(\mathcal{C}\times\left\{  1\right\}  )$. Then Kawan
\cite[Proposition 1.24(ii)]{Kawan13} implies that $\mathbb{P}(D_{0}%
\times\left\{  1\right\}  )$ is contained in a chain control set of the system
restricted to $\mathbb{P}(\mathcal{C}\times\left\{  1\right\}  )$, and hence
is contained in a chain control set of the system on $\mathbb{P}^{n}$. It
follows that $\mathbb{P}(D_{0}\times\left\{  1\right\}  )$ is contained in a
chain control set, hence chain controllable.

(iv) By \cite[Proposition 1.24(iv)]{Kawan13} the lift of a chain control set
is a chain transitive set for the control flow, hence by (iii) the set
$\mathcal{A}$ in (\ref{lift2}) is chain transitive. Furthermore, $\mathcal{A}$
is an invariant compact subset of $\mathcal{U}\times\mathbb{P}^{d,1}$ for the
flow $\mathbb{P}\Phi_{t}^{1}(u,y)=(\theta_{t}u,\mathbb{P}\varphi^{1}(t,y,u))$.

(v) Note that $\mathcal{A}^{\ast}$ is also compact and invariant since
$\mathbb{P}^{d,0}$ is invariant.

\textbf{Claim}: Every $(u,y)\in(\mathcal{U}\times\mathbb{P}^{d})\setminus
(\mathcal{A}\cup\mathcal{A}^{\ast})$ satisfies $\omega(u,y)\subset\mathcal{A}$
and $\alpha(u,y)\subset\mathcal{A}^{\ast}$.

According to Colonius and Kliemann \cite[Lemma B.2.14]{ColK00} the disjoint
compact invariant sets $\mathcal{A}$ and $\mathcal{A}^{\ast}$ form an
attractor-repeller pair$\ $if and only if (a) $(u,y)\not \in \mathcal{A}%
^{\ast}$ implies $\left\{  \mathbb{P}\Phi_{t}^{1}(u,y)\left\vert
t\geq0\right.  \right\}  \cap N\not =\varnothing$ for every neighborhood $N$
of $\mathcal{A}$ and (b) $(u,y)\not \in \mathcal{A}$ implies $\left\{
\mathbb{P}\Phi_{t}^{1}(u,y)\left\vert t\leq0\right.  \right\}  \cap
N\not =\varnothing$ for every neighborhood $N$ of $\mathcal{A}^{\ast}$. Since
the limit sets are nonvoid in the compact space $\mathcal{U}\times
\mathbb{P}^{d}$, these conditions hold if the claim is proved. Thus assertion
(v) will follow.

In order to prove (a), let $(u,y)\not \in \mathcal{A}\cup\mathcal{A}^{\ast}$.
Then $y$ can be written as $y=\mathbb{P}(x,1)$ and by linearity%
\[
\varphi^{1}(t,x,1,u)=(\varphi(t,0,u)+e^{tA}x,1)=\varphi^{1}(t,0,1,u)+\varphi
^{1}(t,x,0,0_{\mathcal{U}}).
\]
Since $e^{At}x\rightarrow0$ for $t\rightarrow\infty$ it follows that%
\begin{align}
\lim_{t\rightarrow\infty}d(\mathbb{P}\Phi_{t}^{1}(u,y),\mathbb{P}\Phi_{t}%
^{1}(u,\mathbb{P}(0,1)))  &  =\lim_{t\rightarrow\infty}d(\left(  \theta
_{t}u,\mathbb{P}\varphi^{1}(t,x,1,u)\right)  ,(\theta_{t}u,\mathbb{P}%
\varphi^{1}(t,0,1,u))\nonumber\\
&  =0. \label{4.8}%
\end{align}
By (i) it follows that $\varphi(t,0,u)\in\mathbf{R}(0)\subset D_{0}$ for all
$t\geq0$. Define $u^{+}\in\mathcal{U}$ by $u^{+}(t):=0$ for $t<0$ and
$u^{+}(t):=u(t)$ for $t\geq0$. Then $(u^{+},\mathbb{P}(0,1))\in\mathcal{A}$
and $d(\mathbb{P}\Phi_{t}^{1}(u^{+},\mathbb{P}(0,1)),\mathbb{P}\Phi_{t}%
^{1}(u,\mathbb{P}(0,1)))\rightarrow0$ for $t\rightarrow\infty$. Together with
(\ref{4.8}) this yields $d(\mathbb{P}\Phi_{t}^{1}(u,y),\mathcal{A}%
)\rightarrow0$ for $t\rightarrow\infty$, hence $\omega(u,y)\subset\mathcal{A}$.

For the proof of (b) again let $(u,y)=(u,\mathbb{P}(x,1))\not \in
\mathcal{A}\cup\mathcal{A}^{\ast}$. We may assume that $\mathbb{P}%
(x,1)\not \in \mathbb{P}(D_{0}\times{1})$. In order to show that
$\alpha(u,y)\subset\mathcal{A}^{\ast}$ it suffices to show that $\Vert
\varphi(-t,x,u)\Vert\rightarrow\infty$ for $t\rightarrow\infty$.

Assume, by contradiction, that $\varphi(-t_{i},x,u)$ remains bounded for a
sequence $t_{i}\rightarrow\infty$. Write $x_{i}=\varphi(-t_{i},x,u)$ and
$u_{i}=u(\cdot-t_{i})$ for all $i\in\mathbb{N}$. Then
\[
x=\varphi(t_{i},x_{i},u_{i})=\varphi(t_{i},x_{i},0)+\varphi(t_{i}%
,0,u_{i})=e^{t_{i}A}x_{i}+\varphi(t_{i},0,u_{i}).
\]
Note that $e^{t_{i}A}x_{i}\rightarrow0$, since $A$ is asymptotically stable
and $\{x_{i},i\in\mathbb{N}\}$ is bounded. Therefore $\varphi(t_{i}%
,0,u_{i})\rightarrow x$ and $x\in\overline{\mathbf{R}(0)}=D_{0}$, which is a contradiction.

(vi) A consequence of (iv), (v), and Theorem \ref{Theorem_attractor} is that
the chain transitive sets of $\mathbb{P}\Phi^{1}$ are either contained in
$\mathcal{U}\times\mathbb{P}^{n,0}$ or coincide with the compact set
$\mathcal{A}\subset\mathcal{U}\times\mathbb{P}^{n,1}$. Thus $\mathcal{A}$ is a
maximal chain transitive set and by Theorem \ref{Theorem_equivalence}
$\mathbb{P}(D_{0}\times\left\{  1\right\}  )$ is a chain control set.
\end{proof}

The following proposition for completely unstable systems is analogous.

\begin{proposition}
\label{PropositionL+}Assume that $\mathbb{R}^{n}=L^{+}$, hence the homogeneous
part is completely unstable. Then the following assertions hold.

(i) The control set $D_{0}$ containing \thinspace$0\in\mathbb{R}^{n}$ is open
and has compact closure satisfying $\overline{D_{0}}=\mathbf{C}(0)$.

(ii) The image $\mathbb{P}(D_{0}\times\left\{  1\right\}  )$ is an open
control set with compact closure%
\[
\overline{\mathbb{P}(D_{0}\times\left\{  1\right\}  )}=\mathbb{P}%
(\overline{D_{0}}\times\left\{  1\right\}  )\subset\mathbb{P}^{n,1}%
\]
for the induced system on $\mathbb{P}^{n}$.

(iii) The set $\overline{\mathbb{P}\left(  D_{0}\times\left\{  1\right\}
\right)  }$ is chain controllable for the induced system on $\mathbb{P}^{n}$.

(iv) The set $\mathcal{A}^{\ast}\subset\mathcal{U}\times\mathbb{P}^{n,1}$
defined by%
\[
\mathcal{A}^{\ast}:=\left\{  (u,\mathbb{P}(x,1))\in\mathcal{U}\times
\mathbb{P}^{n}\left\vert \mathbb{P}(\varphi(t,x,u),1)\in\mathbb{P}(D_{0}%
\times\left\{  1\right\}  )\text{ for all }t\in\mathbb{R}\right.  \right\}
\]
is chain transitive for the control flow $\mathbb{P}\Phi^{1}$.

(v) $\mathcal{A}^{\ast}$ is a repeller and its complementary attractor is
$\mathcal{A}=\mathbb{P}^{n,0}$.

(vi) The set $\overline{\mathbb{P}(D_{0}\times\left\{  1\right\}  }%
)\subset\mathbb{P}^{n,1}$ is the unique chain control set in $\mathbb{P}%
^{n,1}$.
\end{proposition}

\begin{proof}
This follows by time reversal from Proposition \ref{PropositionL-}. Observe
that under by time reversal closed control sets become open control sets and
attractor-repeller pairs exchange their roles.
\end{proof}

Next we combine the previous propositions to obtain the following main result
of this section.

\begin{theorem}
\label{Theorem_ccs1}The unique chain control set $E$ in $\mathbb{R}^{n}$ of
the linear control system (\ref{1}) is given by $E=\overline{D_{0}}+L^{0}$,
where $D_{0}$ is the control set containing $0$ and $L^{0}$ is the center
Lyapunov space of $A$.
\end{theorem}

\begin{proof}
Along with control system (\ref{1}) on $\mathbb{R}^{n}$ consider the induced
control systems on $L^{-}$ and on $L^{+}$, which are described in Lemma
\ref{Lemma_projection}. By Corollary \ref{Corollary_cs} and Lemma
\ref{Lemma_projection}(i) and (ii) the control set $D_{0}$ can be written as
$D_{0}=D^{-}\oplus(L^{0}\cap\mathcal{C})\oplus D^{+}$, where $D^{-}\subset
L^{-}$ is the control set containing $0$ of system (\ref{P}) and $D^{+}\subset
L^{+}$ is the control set containing $0$ of system (\ref{P+}).

By Proposition \ref{PropositionL-}(vii) the unique chain control set in
$\mathbb{P}^{n^{-},1}$ is the compact set $\mathbb{P}(D^{-}\times\left\{
1\right\}  )$. This implies that $\overline{D^{-}}$ is the chain control set
in $L^{-}$. Since distances are decreased under the map $\pi^{-}$ the chain
control set $E$ in $\mathbb{R}^{n}$ satisfies $\pi^{-}E\subset$ $\overline
{D^{-}}$, hence $E\subset\overline{D^{-}}\oplus L^{0}\oplus L^{+}$. By Lemma
\ref{Lemma_projection} $\pi^{-}\mathbf{R}(0)=D^{-}$. It follows that
\[
E\subset\overline{\pi^{-}\mathbf{R}(0)}\oplus L^{0}\oplus L^{+}.
\]
Analogously, one shows that the chain control set $E$ in $\mathbb{R}^{n}$ is
contained in $L\oplus L^{0}\oplus\overline{\pi^{+}\mathbf{C}(0)}$. Using
Theorem \ref{Theorem_cs}(iii) it follows that $E$ is contained in
\[
\left(  \overline{\pi^{-}\mathbf{R}(0)}\oplus L^{0}\oplus L^{+}\right)
\cap\left(  L^{-}\oplus L^{0}\oplus\overline{\pi^{+}\mathbf{C}(0)}\right)
=\overline{\pi^{-}\mathbf{R}(0)}\oplus L^{0}\oplus\overline{\pi^{+}%
\mathbf{C}(0)}=\overline{D_{0}}+L^{0}.
\]
Since by Proposition \ref{Proposition1_2_Eduardo} $\overline{D_{0}}+L^{0}$ is
chain controllable equality holds.
\end{proof}

\begin{remark}
The arguments above provide an alternative proof of the claim that $E$ is the
unique chain control set in $\mathbb{R}^{n}$.
\end{remark}

The following corollary to Theorem \ref{Theorem_ccs1} presents another
description of the chain control set.

\begin{corollary}
For the linear control system (\ref{1}), the chain control set $E$ is%
\[
E=(\overline{\mathbf{C}(0)}\cap L^{+})\oplus L^{0}\oplus(\overline
{\mathbf{R}(0)}\cap L^{-})=(\mathbf{C}^{c}(0)\cap L^{-})\oplus L^{0}%
\oplus(\mathbf{R}^{c}(0)\cap L^{-}).
\]

\end{corollary}

\begin{proof}
Corollary \ref{Corollary_cs} implies%
\[
E=\overline{D_{0}}+L^{0}=\left(  \overline{\mathbf{C}(0)}\cap L^{+}\right)
\oplus L^{0}\oplus(\mathbf{R}(0)\cap L^{-}).
\]
By Proposition \ref{Proposition1_2_Eduardo} the chain reachable set
$\mathbf{R}^{c}(0)$ contains $L^{+}\oplus L^{0}$ and the chain controllable
set $\mathbf{C}^{c}(0)$ contains $L^{0}\oplus L^{-}$. Then decomposition
(\ref{decomposition1}) shows that%
\[
\mathbf{R}^{c}(0)=L^{+}\oplus L^{0}\oplus\left(  \mathbf{R}^{c}(0)\cap
L^{-}\right)  \text{ and }\mathbf{C}^{c}(0)=(\mathbf{C}^{c}(0)\cap
L^{+})\oplus L^{0}\oplus L^{-}.
\]
Combining this with Theorem \ref{Theorem_control_inv} one finds%
\[
E=\mathbf{C}^{c}(0)\cap\mathbf{R}^{c}(0)=(\mathbf{C}^{c}(0)\cap L^{-})\oplus
L^{0}\oplus(\mathbf{R}^{c}(0)\cap L^{-}).
\]

\end{proof}

Comparing this to the characterization of the control set $D_{0}$ in Corollary
\ref{Corollary_cs} one sees that the difference between the control set
$D_{0}$ and the chain control set $E$ is that $D_{0}$ contains the summand
$L^{0}\cap\mathcal{C}$ while for $E$ the whole center subspace $L^{0}$ is a
summand and the closure of $\mathbf{C}(0)\cap L^{+}$ is taken.

\section{The system on the Poincar\'{e} sphere\label{Section5}}

In this section we deduce a formula for the central Selgrade bundle and show
that the chain control set on the Poincar\'{e} sphere coincides with the
closure of the image of the chain control set on $\mathbb{R}^{n}$.

By Theorem \ref{Theorem_Selgrade1} we know that the central Selgrade bundle
$\mathcal{V}_{c}^{1}$ for (\ref{1}) is the unique Selgrade bundle not
contained in $\mathcal{U}\times\mathbb{R}^{n+1,0}$; cf. (\ref{4.2}). The
projection $\mathcal{M}_{c}^{1}=\mathbb{P}\mathcal{V}_{c}^{1}\subset
\mathcal{U}\times\mathbb{P}^{n}$ is the unique chain recurrent component of
$\mathbb{P}\Phi^{1}$ with $\mathcal{M}_{c}^{1}\cap\left(  \mathcal{U}%
\times\mathbb{P}^{n,1}\right)  \not =\varnothing$. Furthermore,
\[
E_{c}^{1}=\{p\in\mathbb{P}^{n}\left\vert \exists u\in\mathcal{U}%
:(u,p)\in\mathcal{M}_{c}^{1}\right.  \}
\]
is the unique chain control set not contained in the equator $\mathbb{P}%
^{n,0}$ of the Poincar\'{e} sphere $\mathbb{P}^{n}$, i.e., having nonvoid
intersection with $\mathbb{P}^{n,1}$.

We will exploit the decomposition of $\mathbb{R}^{n}$ into the hyperbolic part
$L^{+}\oplus L^{-}$ and $L^{0}$. Fixing a basis, we identify the subspace
$L^{+}\oplus L^{-}$ with $\mathbb{R}^{n^{h}}$ where $n^{h}=\dim(L^{+}\oplus
L^{-})$. The corresponding Poincar\'{e} sphere is $\mathbb{P}^{n^{h}}$. The
next proposition provides basic information on the behavior of the induced
control system (\ref{hyp}) on $L^{+}\oplus L^{-}$.

\begin{proposition}
\label{Proposition_hyp1}(i) The chain control set $E^{h}$ for the induced
control system (\ref{hyp}) on $L^{+}\oplus L^{-}$ is compact and
$E^{h}=\overline{D_{0}^{h}}$, where $D_{0}^{h}=\pi^{h}D_{0}$ is the control
set containing $0.$

(ii) There exists a unique chain control set $E_{c}^{h}$ having nonvoid
intersection with $\mathbb{P}^{n^{h},1}$ for the induced control system on
$\mathbb{P}^{n^{h}}$. It is a compact subset of $\mathbb{P}^{n^{h},1}$ and
coincides with $\overline{\mathbb{P}\left(  E^{h}\times\left\{  1\right\}
\right)  }=\overline{\mathbb{P}\left(  D_{0}^{h}\times\left\{  1\right\}
\right)  }$.

(iii) Consider the flows $\mathbb{P}\Phi^{1}$ on $\mathcal{U}\times
\mathbb{P}^{n}$ and $\mathbb{P}\Phi^{h,1}$ on $\mathcal{U}\times
\mathbb{P}^{n^{h}}$. The chain recurrent components $\mathcal{M}_{c}%
^{1}=\mathbb{P}\mathcal{V}_{1}^{c}$ and $\mathcal{M}_{c}^{h,1}=\mathbb{P}%
\mathcal{V}_{c}^{h,1}$ with $\mathcal{M}_{c}^{1}\cap\mathbb{P}^{n,1}%
\not =\varnothing$ and $\mathcal{M}_{c}^{h,1}\cap\mathbb{P}^{n^{h},1}%
\not =\varnothing$ are the lifts $\mathcal{E}_{c}^{1}$ and $\mathcal{E}%
_{c}^{h,1}$ of the chain control sets $E_{c}^{1}$ and $E_{c}^{h,1}$, respectively.

(iv) The central Selgrade bundle for the flow $\Phi^{h,1}$ on $\mathcal{U}%
\times\left(  L^{+}\oplus L^{-}\right)  \times\mathbb{R}$ is the line bundle%
\[
\mathcal{V}_{c}^{h,1}=\left\{  (u,-re(u,0),r)\in\mathcal{U}\times
\mathbb{R}^{n^{h}}\times\mathbb{R}\left\vert u\in\mathcal{U},r\in
\mathbb{R}\right.  \right\}  ,
\]
where $e(u,t),t\in\mathbb{R}$, is the unique bounded solution of (\ref{hyp})
for $u\in\mathcal{U}$, and the projection satisfies $\mathcal{M}_{c}%
^{h,1}=\mathbb{P}\mathcal{V}_{c}^{h,1}\subset\mathbb{P}^{n^{h},1}$.
\end{proposition}

\begin{proof}
Assertion (i) follows from the characterization of the chain control set in
Theorem \ref{Theorem_ccs1} using that $A^{h}$ is hyperbolic, and by Lemma
\ref{Lemma_projection}(iii). Assertion (iii) follows by (i) and Corollary
\ref{Corollary_splitE}(iii). Furthermore (iii) holds, since the lifts of chain
control sets are chain recurrent components by Theorem
\ref{Theorem_equivalence}, and (iv) holds by Theorem \ref{Theorem_Selgrade1}(iv).
\end{proof}

Define a map of the projective Poincar\'{e} spheres by%
\begin{equation}
P:\mathbb{P}^{n}\rightarrow\mathbb{P}^{n^{h}}:P(\mathbb{P}(x,r)):=\mathbb{P(}%
\pi^{h}x,r)\text{ for }(0,0)\not =(x,r)\in\mathbb{R}^{n}\times\mathbb{R}.
\label{P1}%
\end{equation}

\begin{proposition}
\label{Proposition_hyp2}(i) The chain control sets $E_{c}^{1}$ and
$E_{c}^{h,1}$ of the induced control systems on the Poincar\'{e} spheres
$\mathbb{P}^{n}$ and $\mathbb{P}^{n^{h}}$, respectively, satisfy%
\[
P(E_{c}^{1})\subset E_{c}^{h,1}=\overline{\mathbb{P}\left(  E^{h}%
\times\left\{  1\right\}  \right)  }=\overline{\mathbb{P}\left(  D_{0}%
^{h}\times\left\{  1\right\}  \right)  }.
\]

(ii) The central Selgrade bundles $\mathcal{V}_{c}^{1}$ and $\mathcal{V}%
_{c}^{h,1}$ of the flow $\Phi^{1}$ on $\mathcal{U}\times\mathbb{R}^{n}%
\times\mathbb{R}$ and the flow $\Phi^{h,1}$ on $\mathcal{U}\times\left(
L^{+}\oplus L^{-}\right)  \times\mathbb{R}$, respectively, satisfy%
\[
\mathcal{V}_{c}^{1}\subset\left\{  (u,x,r)\in\mathcal{U}\times\mathbb{R}%
^{n}\times\mathbb{R}\left\vert (u,\pi^{h}x,r)\in\mathcal{V}_{c}^{h,1}\right.
\right\}  .
\]

\end{proposition}

\begin{proof}
(i) The map $\mathbb{R}^{n}\times\mathbb{R}\rightarrow\left(  L^{+}\oplus
L^{-}\right)  \times\mathbb{R}:(x,r):=\mathbb{(}\pi^{h}x,r)$ is a projection
decreasing the norm, hence the map $P$ decreases the distance. Thus controlled
$(\varepsilon,T)$-chains in $\mathbb{P}^{n}$ are mapped to controlled
$(\varepsilon,T)$-chains in $\mathbb{P}^{n^{h}}$ showing that the image
$P(E_{c}^{1})$ is contained in a chain control set. Since $E_{c}^{1}%
\cap\mathbb{P}^{n,1}\not =\varnothing$ it follows that this chain control set
has nonvoid intersection with $\mathbb{P}^{n^{h},1}$, hence coincides with
$E_{c}^{h,1}$. The equalities hold by Proposition \ref{Proposition_hyp1}(ii).

(ii) This follows from (i) and Proposition \ref{Proposition_hyp1}(iii)..
\end{proof}

The following theorem characterizes the central Selgrade bundle $\mathcal{V}%
_{c}^{1}$ and, equivalently, the chain control set $E_{c}^{1}$ on the
projective Poincar\'{e} sphere $\mathbb{P}^{n}$. Recall that $\mathcal{V}%
_{c}^{h,1}$ is the central Selgrade bundle of the hyperbolic part and
$\mathcal{V}_{0}^{\infty}=\mathcal{U}\times L^{0}\times0$ with the center
Lyapunov space $L^{0}=L(0)$.

\begin{theorem}
\label{Theorem_central}(i) The central Selgrade bundle $\mathcal{V}_{c}^{1}$
of the control flow $\Phi^{1}$ on $\mathcal{U}\times\mathbb{R}^{n}%
\times\mathbb{R}$ associated with control system (\ref{1}) is given by%
\begin{equation}
\mathcal{V}_{c}^{1}=\mathcal{V}_{c}^{h,1}\oplus\mathcal{V}_{0}^{\infty
}=\left\{  (u,-re(u,0)+L^{0},r)\in\mathcal{U}\times\mathbb{R}^{n}%
\times\mathbb{R}\left\vert u\in\mathcal{U},r\in\mathbb{R}\right.  \right\}  ,
\label{V_c}%
\end{equation}
where $e(u,t),t\in\mathbb{R}$ is the unique bounded solution for $u$ of the
induced system (\ref{hyp}) on $L^{+}\oplus L^{-}$. The dimension is
$\dim\mathcal{V}_{c}^{1}=1+\dim L^{0}$.

(ii) The chain control set $E_{c}^{1}$ on the Poincar\'{e} sphere
$\mathbb{P}^{n}$ is given by%
\[
E_{c}^{1}=\overline{\left\{  \mathbb{P}\left(  -e(u,0)+L^{0},1\right)
\left\vert u\in\mathcal{U}\right.  \right\}  }.
\]

\end{theorem}

\begin{proof}
Assertion (ii) is a consequence of (i). The second equality in (i) holds since
by Proposition \ref{Proposition_hyp1}(iv) the central Selgrade bundle
$\mathcal{V}_{c}^{h,1}$ of the hyperbolic part is%
\[
\mathcal{V}_{c}^{h,1}=\left\{  (u,-re(u,0),r)\in\mathcal{U}\times
\mathbb{R}^{n^{h}}\times\mathbb{R}\left\vert u\in\mathcal{U},r\in
\mathbb{R}\right.  \right\}  .
\]
It remains to prove the first equality in (i).

\textbf{Step 1.} We claim that%
\[
\mathcal{V}^{\prime}:=\mathcal{V}_{c}^{h,1}\oplus\mathcal{V}_{0}^{\infty
}=\left\{  (u,-re(u,0)+L^{0},r)\in\mathcal{U}\times\mathbb{R}^{n}%
\times\mathbb{R}\left\vert u\in\mathcal{U},r\in\mathbb{R}\right.  \right\}
\]
defines a subbundle. This holds if it is closed and the fibers are linear with
constant dimension; cf. Colonius and Kliemann \cite[Lemma B.1.13]{ColK00}. In
fact, consider for $u\in\mathcal{U}$ the fiber%
\[
V_{u}^{\prime}=\left\{  (-re(u,0)+rx,r)\in\mathbb{R}^{n}\times\mathbb{R}%
\left\vert x\in L^{0},r\in\mathbb{R}\right.  \right\}  ,
\]
and let $(-r_{i}e(u,0)+x_{i},r_{i})\in V_{u}^{\prime}$ for $i=1,2$. Then for
all $\alpha,\beta\in\mathbb{R}$%
\begin{align*}
&  \alpha(-r_{1}e(u,0)+x_{1},r_{1})+\beta(-r_{2}e(u,0)+x_{2},r_{2})\\
&  =(-(\alpha r_{1}+\beta r_{2})e(u,0)+(\alpha x_{1}+\beta x_{2}),\alpha
r_{1}+\beta r_{2})\in V_{u}^{\prime}.
\end{align*}
The dimension is $\dim V_{u}^{\prime}=1+\dim L^{0}$. Furthermore, if
$(u_{k},-r_{k}e(u_{k},0)+x_{k},r_{k})\rightarrow(v,y,r)$ for $k\rightarrow
\infty$, it follows that $u_{k}\rightarrow u$ and $r_{k}\rightarrow r$. Then
(cf. Colonius and Santana \cite[Theorem 2.5 and proof of Theorem
3.1]{ColSan11}) it follows that the unique bounded solutions at time $t=0$
given by $e(u_{k},0)$ converge to $e(u,0)$, hence $-r_{k}e(u_{k}%
,0)+x_{k}\rightarrow re(u,0)+x$. This shows that $(v,y,r)=(u,-re(u,0)+rx,r)\in
V_{u}^{\prime}$, hence $\mathcal{V}^{\prime}$ a subbundle.

\textbf{Step 2. }The subbundles satisfy $\mathcal{V}_{c}^{1}\subset
\mathcal{V}^{\prime}$.

Write $x=\pi^{h}x+\pi^{0}x$ for $x\in\mathbb{R}^{n}$. By Proposition
\ref{Proposition_hyp2}(ii) every $(u,x,r)\in\mathcal{V}_{c}^{1}$ satisfies%
\[
(u,\pi^{h}x,r)\in\mathcal{V}_{c}^{h,1}=\left\{  (u,-re(u,0),r)\in
\mathcal{U}\times(L^{+}\oplus L^{-})\times\mathbb{R}\left\vert u\in
\mathcal{U},r\in\mathbb{R}\right.  \right\}
\]
implying $(u,x,r)=(u,\pi^{h}x+\pi^{0}x,r)=(u,-re(u,0)+\pi^{0}x,r)\in
\mathcal{V}^{\prime}.$

\textbf{Step 3. }Also the converse inclusion $\mathcal{V}^{\prime}%
\subset\mathcal{V}_{c}^{1}$ holds, hence $\mathcal{V}^{\prime}=\mathcal{V}%
_{c}^{1}$.

For the proof let $(u,y,r)\in\mathcal{V}^{\prime}=\mathcal{V}_{c}^{h,1}%
\oplus\mathcal{V}_{0}^{\infty}$. Observe first that by Proposition
\ref{Proposition_in} the inclusion $\mathcal{V}_{0}^{\infty}\subset
\mathcal{V}_{c}^{1}$ holds. If%
\begin{align}
\mathcal{V}_{c}^{h,1}  &  =\left\{  (u,-re(u,0),r)\in\mathcal{U}%
\times\mathbb{R}^{n}\times\mathbb{R}\left\vert r\in\mathbb{R}\right.  \right\}
\nonumber\\
&  \subset\left\{  (u,\pi^{h}x,r)\in\mathcal{U}\times\mathbb{R}^{n}%
\times\mathbb{R}\left\vert (u,x,r)\in\mathcal{V}_{c}^{1}\right.  \right\}
=:\pi^{h}\mathcal{V}_{c}^{1}, \label{REST_1}%
\end{align}
the assertion follows since by linearity this implies%
\[
(u,y,r)=(u,\pi^{h}y+\pi^{0}y,r)\in\mathcal{V}_{c}^{h,1}\oplus\mathcal{V}%
_{0}^{\infty}\subset\pi^{h}\mathcal{V}_{c}^{1}\oplus\mathcal{V}_{0}^{\infty
}\subset\mathcal{V}_{c}^{1}.
\]
Inclusion (\ref{REST_1}) is equivalent to
\begin{equation}
\mathcal{M}_{c}^{h,1}\subset\left\{  (u,\mathbb{P}\left(  \pi^{h}x,1\right)
)\in\mathcal{U}\times\mathbb{P}^{n}\left\vert (u,\mathbb{P}\left(  x,r\right)
)\in\mathcal{M}_{c}^{1}\right.  \right\}  . \label{REST_2}%
\end{equation}
By Proposition \ref{Proposition_hyp1}(iii) the chain transitive components
$\mathcal{M}_{c}^{h,1}$ and $\mathcal{M}_{c}^{1}$ are the lifts $\mathcal{E}%
_{c}^{h,1}$ and $\mathcal{E}_{c}^{1}$ of the chain control sets $E_{c}^{h,1}$
and $E_{c}^{1}$, respectively. Thus (\ref{REST_2}) is, with $P:\mathbb{P}%
^{n}\rightarrow\mathbb{P}^{n^{h}}$ defined by (\ref{P1}), equivalent to%
\begin{equation}
E_{c}^{h,1}\subset P(E_{c}^{1})=\left\{  \mathbb{P(}\pi^{h}x,1)\left\vert
\mathbb{P}(x,1)\in E_{c}^{1}\right.  \right\}  . \label{REST_3}%
\end{equation}
For the proof of inclusion (\ref{REST_3}) recall that by Corollary
\ref{Corollary_splitE}(iii) $E_{c}^{h,1}=\overline{\mathbb{P(}E^{h}%
\times\left\{  1\right\}  )}$ and by Proposition \ref{Proposition_hyp1}(i) the
chain control set $E^{h}$ and the control sets $D_{0}^{h}$ and $D_{0}$ are
related by%
\[
E^{h}=\overline{D_{0}^{h}},\text{ and }\pi^{h}(D_{0}\mathbb{)}=D_{0}%
^{h},\text{ hence }E^{h}=\overline{\pi^{h}(D_{0}\mathbb{)}}.
\]
It follows that%
\begin{equation}
E_{c}^{h,1}=\overline{\mathbb{P(}E^{h}\times\left\{  1\right\}  )}%
=\overline{\mathbb{P(}\overline{\pi^{h}(D_{0}\mathbb{)}}\times\left\{
1\right\}  )}. \label{R4}%
\end{equation}
The points in $\left(  \mathrm{int}D_{0}\right)  \times\left\{  1\right\}
\subset\mathbb{R}^{n}\times\mathbb{R}$ are controllable to each other. Hence
the same is true for the points in $\mathbb{P}\left(  \left(  \mathrm{int}%
D_{0}\right)  \times\left\{  1\right\}  \right)  \subset\mathbb{P}^{n}$ and it
follows that $\mathbb{P}\left(  D_{0}\times\left\{  1\right\}  \right)
\subset E_{c}^{1}$, thus $x\in D_{0}$ implies $\mathbb{P}(x,1)\in E_{c}^{1}$.
This shows that%
\[
\mathbb{P(}\pi^{h}(D_{0}\mathbb{)}\times\left\{  1\right\}  )=\left\{
\mathbb{P(}\pi^{h}x,1)\left\vert x\in D_{0}\right.  \right\}  \subset\left\{
\mathbb{P(}\pi^{h}x,1)\left\vert \mathbb{P}(x,1)\in E_{c}^{1}\right.
\right\}  .
\]
Taking closures and using (\ref{R4}) inclusion (\ref{REST_3}) follows. This
concludes the proof of Theorem \ref{Theorem_central}.
\end{proof}

Finally we compare the image $\mathbb{P}(E\times\left\{  1\right\}  )$ of the
chain control set in $\mathbb{R}^{n}$ and the chain control set $E_{c}^{1}$ on
the Poincar\'{e} sphere on $\mathbb{P}^{n}$.

\begin{corollary}
\label{Corollary_ccs}The chain control set $E_{c}^{1}$ on the Poincar\'{e}
sphere coincides with the closure of the image of the chain control set $E$ in
$\mathbb{R}^{n}$, $E_{c}^{1}=\overline{\mathbb{P}(E\times\left\{  1\right\}
)}$.
\end{corollary}

\begin{proof}
By Corollary \ref{Corollary_splitE}(ii) we already know that $\mathbb{P}%
(E\times\left\{  1\right\}  )\subset E_{c}^{1}$- Since chain control sets are
closed, it only remains to prove the converse inclusion. First we show that
the unique bounded solution of (\ref{hyp}) for control $u\in\mathcal{U}$
satisfies $e(u,\cdot)\subset E^{h}$, where $E^{h}$ is the chain control set of
(\ref{hyp}). Consider the set%
\[
\left\{  y\in\mathbb{R}^{n}\left\vert \exists t_{k}\rightarrow\infty
:\varphi^{h}(t_{k},e(u,0),u)=e(u,t_{k})\rightarrow y\right.  \right\}  ,
\]
where $\varphi^{h}(t,x,u),t\in\mathbb{R}$, denotes the solution of
(\ref{hyp}). This set is nonvoid since $e(u,\cdot)$ is bounded and, as an
$\omega$-limit set it is contained in a chain control set, hence in $E^{h}$.
Analogously,%
\[
\varnothing\not =\left\{  y\in\mathbb{R}^{n}\left\vert \exists t_{k}%
\rightarrow-\infty:\varphi^{h}(t_{k},e(u,0),u)\rightarrow y\right.  \right\}
\subset E^{h}.
\]
It follows that $e(u,t),t\in\mathbb{R}$, is contained in $E^{h}$. By Corollary
\ref{Corollary_cs} we know for the control sets $D_{0}\subset\mathbb{R}^{n}$
and $D_{0}^{h}\subset L^{+}\oplus L^{-}$ containing the origin that
\[
D_{0}=\overline{\mathbf{R}(0)}\cap\mathbf{C}(0)\text{ and }D_{0}^{h}%
=\overline{\mathbf{R}^{h}(0)}\cap\mathbf{C}^{h}(0).
\]
By Lemma \ref{Lemma_projection}(ii) $\mathbf{R}(0)\subset\mathbf{R}%
^{h}(0)+L^{0}$ and hence $\mathbf{R}(0)+L^{0}=\mathbf{R}^{h}(0)+L^{0}$.
Analogously it follows that $\mathbf{C}(0)+L^{0}=\mathbf{C}^{h}(0)+L^{0}$, and
hence%
\[
D_{0}+L^{0}=\left(  \overline{\mathbf{R}(0)}\cap\mathbf{C}(0)\right)
+L^{0}=\left(  \overline{\mathbf{R}^{h}(0)}+L^{0}\right)  \cap\left(
\mathbf{C}^{h}(0)+L^{0}\right)  =D_{0}^{h}+L^{0}.
\]
By Theorem \ref{Theorem_ccs1} the chain control sets are $E=\overline{D_{0}%
}+L^{0}$ and $E^{h}=\overline{D_{0}^{h}}$ and we obtain%
\[
E=\overline{D_{0}}+L^{0}=\overline{D_{0}^{h}}+L^{0}=E^{h}+L^{0}.
\]
By Theorem \ref{Theorem_central}(ii) it follows that%
\[
E_{c}^{1}=\overline{\mathbb{P}\left\{  \left(  -e(u,0)+L^{0},1\right)
\left\vert u\in\mathcal{U}\right.  \right\}  }\subset\overline{\mathbb{P}%
\left\{  \left(  E^{h}+L^{0}\right)  \times\left\{  1\right\}  \right\}
}=\overline{\mathbb{P}\left(  E\times\left\{  1\right\}  \right)  }.
\]
This concludes the proof of the corollary.
\end{proof}

\section{Examples\label{Section6}}

In this section we present two examples illustrating the results above. In the
first example the matrix $A$ is nonhyperbolic and the controllability subspace
is a proper subspace.

\begin{example}
Consider the two dimensional system%
\[
\left(
\begin{array}
[c]{c}%
\dot{x}\\
\dot{y}%
\end{array}
\right)  =\left(
\begin{array}
[c]{cc}%
0 & 0\\
0 & -1
\end{array}
\right)  \left(
\begin{array}
[c]{c}%
x\\
y
\end{array}
\right)  +\left(
\begin{array}
[c]{c}%
0\\
1
\end{array}
\right)  u(t),\,u(t)\in\lbrack-1,1].
\]
The controllability subspace is $\mathcal{C}=\operatorname{Im}[B,AB]=0\times
\mathbb{R}$. The control set $D_{0}$ containing $0\in\mathbb{R}^{2}$ is given
by $D_{0}=0\times\left[  -1,1\right]  \subset\mathcal{C}=0\times\mathbb{R}$
with nonvoid interior in $\mathcal{C}$. The center subspace of $A$ is
$L^{0}=L(0)=\mathbb{R}\times0$ and by Theorem \ref{Theorem_ccs1} the chain
control set is%
\[
E=\overline{D_{0}}+L^{0}=\left(  0\times\left[  -1,1\right]  \right)  +\left(
\mathbb{R}\times0\right)  =\mathbb{R}\times\lbrack-1,1].
\]
Consider the subbundles $\mathcal{V}_{0}=\mathcal{U}\times L(0)=\mathcal{U}%
\times\left(  \mathbb{R}\times0\right)  $ and (cf. Proposition
\ref{Proposition_hyp1}(iv))%
\[
\mathcal{V}_{c}^{h,1}=\left\{  (u,-re(u,0),r)\in\mathcal{U}\times
\mathbb{R}\times\mathbb{R}\left\vert u\in\mathcal{U},r\in\mathbb{R}\right.
\right\}  ,
\]
where $e(u,t),t\in\mathbb{R}$, is the unique bounded solution for
$u\in\mathcal{U}$ of the hyperbolic part. By Theorem \ref{Theorem_central}(i)
the central Selgrade bundle is%
\[
\mathcal{V}_{c}^{1}=\mathcal{V}_{c}^{h,1}\oplus\mathcal{V}_{0}^{\infty
}=\left\{  (u,-re(u,0)+x,0,r)\left\vert u\in\mathcal{U},x\in\mathbb{R}%
,r\in\mathbb{R}\right.  \right\}  ,
\]
A little computation shows that the unique bounded solution of $\dot
{y}(t)=-y(t)+u(t)$ is%
\begin{equation}
e(u,t):=\int_{-\infty}^{t}e^{-(t-s)}u(s)ds,t\in\mathbb{R}. \label{bounded}%
\end{equation}
Thus it follows that%
\[
\mathcal{V}_{c}^{1}=\left\{  \left.  \left(  u,x,-r\int_{-\infty}^{0}%
e^{s}u(s)ds,r\right)  \right\vert u\in\mathcal{U},x\in\mathbb{R}%
,r\in\mathbb{R}\right\}  \subset\mathcal{U}\times\mathbb{R}^{2}.
\]
Corollary \ref{Corollary_ccs} shows that the chain control set $E_{c}^{1}$ on
the projective Poincar\'{e} sphere is%
\[
E_{c}^{1}=\mathbb{P}\left(  E\times\left\{  1\right\}  \right)  =\mathbb{P}%
\left(  \mathbb{R}\times\lbrack-1,1]\times\left\{  1\right\}  \right)
\subset\mathbb{P}^{2}.
\]

\end{example}

The next example is controllable and the matrix $A$ is hyperbolic.

\begin{example}
Consider%
\begin{equation}
\left(
\begin{array}
[c]{c}%
\dot{x}\\
\dot{y}%
\end{array}
\right)  =\left(
\begin{array}
[c]{cc}%
1 & 0\\
0 & -1
\end{array}
\right)  \left(
\begin{array}
[c]{c}%
x\\
y
\end{array}
\right)  +\left(
\begin{array}
[c]{c}%
1\\
1
\end{array}
\right)  u(t),u(t)\in\lbrack-1,1]. \label{6.2}%
\end{equation}
The controllability subspace is $\mathcal{C}=\mathbb{R}^{2}$ and inspection of
the phase portrait shows that the control set is $D_{0}=\left(  -1,1\right)
\times\left[  -1,1\right]  $. By Theorem \ref{Theorem_ccs1} the chain control
set is%
\[
E=\overline{D_{0}}+L^{0}=\overline{D_{0}}=[-1,1]\times\lbrack-1,1].
\]
By Theorem \ref{Theorem_central} the central Selgrade bundle is%
\[
\mathcal{V}_{c}^{1}=\mathcal{V}_{c}^{h,1}=\left\{  (u,-re(u,0),r)\left\vert
x\in L^{0},r\in\mathbb{R}\right.  \right\}  \subset\mathcal{U}\times
\mathbb{R}^{2},
\]
where $e(u,t),t\in\mathbb{R}$, is the bounded solution of (\ref{6.2}) for
$u\in\mathcal{U}$. The bounded solution of $\dot{y}(t)=-y(t)+u(t)$ is given by
(\ref{bounded}). The bounded solution of $\dot{x}(t)=x(t)+u(t)$ is
$x_{b}(t):=-\int_{-\infty}^{-t}e^{t+s}u(-s)ds,t\in\mathbb{R}$. Thus
\[
\mathcal{V}_{c}^{1}=\left\{  \left.  \left(  u,r\int_{-\infty}^{0}%
e^{s}u(-s)ds,-r\int_{-\infty}^{0}e^{s}u(s)ds,r\right)  \right\vert
u\in\mathcal{U},r\in\mathbb{R}\right\}  .
\]
Corollary \ref{Corollary_ccs} shows that the chain control set $E_{c}^{1}$ on
the projective Poincar\'{e} sphere is%
\[
E_{c}^{1}=\left\{  \mathbb{P}\left.  \left(  -\int_{-\infty}^{0}%
e^{s}u(-s)ds,-\int_{-\infty}^{0}e^{s}u(s)ds,1\right)  \right\vert
u\in\mathcal{U}\right\}  =\mathbb{P}\left(  E\times\left\{  1\right\}
\right)  .
\]
\bigskip
\end{example}

\end{document}